\documentclass[11pt,reqno]{amsart}
\usepackage{fullpage}
\usepackage{a4wide}
\usepackage[utf8x]{inputenc}
\usepackage[T1]{fontenc}
\usepackage[english]{babel}
\usepackage{hyperref}
\usepackage{amsfonts,amsmath,amssymb,amsthm,amsrefs}
\usepackage{latexsym}
\usepackage{layout}
\usepackage{dsfont}
\usepackage{color}

\newtheorem{theorem}{Theorem}

\newtheorem{lemma}[theorem]{Lemma}

\theoremstyle{definition}

\theoremstyle{remark}
\newtheorem{remark}[theorem]{Remark}

\newcommand{\p}{\mathbb{P}}
\newcommand{\e}{\mathbb{E}}

\newcommand{\reals}{\mathbb{R}}
\newcommand{\ind}{\mathbf{1}}

\newcommand{\me}{\mathrm{e}}
\newcommand{\md}{\mathrm{d}}

\newcommand{\drift}{c}
\newcommand{\WW}{\mathbb{W}}
\newcommand{\qscale}{W^{(q)}}
\newcommand{\pscale}{W^{(p)}}

\newcommand{\wq}{w^{(q)}}

\def\beq{\begin{eqnarray}} \def\eeq{\end{eqnarray}}

\def\al*#1{\begin{align*}#1\end{align*}}

\def\ga*#1{\begin{gather*}#1\end{gather*}}

\def\alat*#1#2{\begin{alignat*}{#1}#2\end{alignat*}}
\def\bea{\begin{eqnarray*}}
\def\eea{\end{eqnarray*}}
\def\ml*#1{\begin{multline*}#1\end{multline*}}

\allowdisplaybreaks

\begin{document}

\title[]{Parisian ruin for a refracted L\'evy process}

\author[Lkabous \& Czarna \& Renaud]{Mohamed Amine Lkabous}
\address{D\'epartement de math\'ematiques, Universit\'e du Qu\'ebec \`a Montr\'eal (UQAM), 201 av.\ Pr\'esident-Kennedy, Montr\'eal (Qu\'ebec) H2X 3Y7, Canada}
\email{lkabous.mohamed\_amine@courrier.uqam.ca}
\author[]{Irmina Czarna}
\address{Department of Mathematics, University of Wroc\l aw, pl. Grunwaldzki 2/4, 50-384 Wroc\l aw, Poland}
\email{czarna@math.uni.wroc.pl}
\author[]{Jean-Fran\c{c}ois Renaud}
\address{D\'epartement de math\'ematiques, Universit\'e du Qu\'ebec \`a Montr\'eal (UQAM), 201 av.\ Pr\'esident-Kennedy, Montr\'eal (Qu\'ebec) H2X 3Y7, Canada}
\email{renaud.jf@uqam.ca}
 
\date{\today}

\begin{abstract}
In this paper, we investigate Parisian ruin for a L\'evy surplus process with an adaptive premium rate, namely a refracted L\'evy process. Our main contribution is a generalization of the result in \cite{loeffenetal2013} for the probability of Parisian ruin of a standard L\'evy insurance risk process. More general Parisian boundary-crossing problems with a deterministic implementation delay are also considered. Despite the more general setup considered here, our main result is as compact and has a similar structure. Examples are provided.
\end{abstract}


\maketitle


\section{Introduction}

In the last few years, the idea of Parisian ruin has attracted a lot of attention. In Parisian-type ruin models, the insurance company is not immediately liquidated when it defaults: a grace period is granted before liquidation. More precisely, Parisian ruin occurs if the time spent below a pre-determined critical level is longer than the implementation delay, also called the \textit{clock}. Originally, two types of Parisian ruin have been considered, one with deterministic delays (see e.g.\ \cites{czarnapalmowski2010,loeffenetal2013,Wong_Cheung2015,landriaultetal2017}) and another one with stochastic delays (\cites{landriaultetal2011,landriaultetal2014,baurdoux_et_al_2015}). These two types of Parisian ruin start a new clock each time the surplus enters the \textit{red zone}, either deterministic or stochastic. A third definition of Parisian ruin, called cumulative Parisian ruin, has been proposed very recently in \cite{guerin_renaud_2015}; in that case, the \textit{race} is between a single deterministic clock and the sum of the excursions below the critical level.

In this paper, we are interested in the time of Parisian ruin with a deterministic delay for a refracted L\'evy insurance risk process, a process first studied in \cite{kyprianouloeffen2010}. For a standard L\'evy insurance risk process $X$, the time of Parisian ruin, with delay $r>0$, has been studied in \cite{loeffenetal2013}: it is defined as
$$
\kappa_r = \inf \left\lbrace t > 0 \colon t - g_t > r \right\rbrace ,
$$
where $g_t = \sup \left\lbrace 0 \leq s \leq t \colon X_s \geq 0 \right\rbrace$. Loeffen et al.\ \cite{loeffenetal2013} obtained a very nice and compact expression for the probability of Parisian ruin:
\begin{theorem}
For $x \in \reals$,
\begin{equation}\label{SNLPPr}
\p_{x}\left(\kappa_{r}<\infty \right) = 1-\left(\e[X_{1}]\right)_{+}\frac{\int^{\infty}_{0}W(x+z)z\p(X_{r}\in \md z)}{\int^{\infty}_{0}z\p(X_{r}\in \md z)} ,
\end{equation}
where $(x)_{+}=\max(x,0)$ and where the function $W$ is the $0$-scale function of $X$ (see its definition in~\eqref{def_scale})
\end{theorem}

We want to improve on this result by making the model more general and realistic, as suggested in \cite{renaud2014}, by using a process with adaptive premium for the surplus process. More precisely, when the company is in financial distress, that is when its surplus is below the critical level, the premium is increased; and when its surplus leaves that \textit{red zone} then the premium is brought back to its regular level. Therefore, we will use a refracted L\'evy process as our surplus process. 

%

Note that we could also interpret this change in the premium rate as a way to invest (for R\&D, modernization, etc.): if the surplus of the company is in a good financial situation, i.e.\ above the \textit{critical level}, then it invests at rate $\delta$; otherwise it does not. However, for the rest of this paper, we will use the previous interpretation.

In general, fluctuation identities for refracted L\'evy processes can be tedious compared to their classical counterparts because scale functions of two different L\'evy risk processes are involved (see \cite{kyprianouloeffen2010}). Therefore, our main contribution is a surprisingly compact expression for the probability of Parisian ruin for a refracted L\'evy risk process (see Equation~\eqref{refractedPr} below), in the spirit of the one in Equation~\eqref{SNLPPr} for a standard L\'evy risk process. Our formula also provides information on how the \textit{refraction parameter} affects this probability while displaying the impact of the \textit{delay parameter}. Moreover, we analyze more general Parisian boundary-crossing problems for the refracted L\'evy process which have not been studied previously, even for a standard L\'evy risk process. As a consequence, when the \textit{refraction parameter} it set to zero, new identities for the classical L\'evy setup are obtained.

The rest of the paper is organized as follows. In Section 2, we present our model in more details together with some background material on spectrally negative L\'evy processes and scale functions. The main results are presented in Section 3, while Section 4 presents a few examples. Section 5 is devoted to the proofs of the main results as well as (new) technical lemmas. In the Appendix, a few well known properties of scale functions are presented.

\section{Our model and background material}

As mentioned in the introduction, we are interested in a surplus process $U$ whose dynamics change by adding a fixed linear drift (premium) whenever it is below the critical level, a region also called the \textit{red zone}. Without loss of generality, we will choose this critical level to be $0$.

In our model, $Y$ is the surplus process during \textit{regular business periods} (above zero), while $X$ is the surplus process, with an additional rate of premium $\delta$, for \textit{critical business periods} (below zero). More precisely, let $Y$ be a Lévy insurance risk process (see the definition below) modelling the dynamic of the surplus $U$ above $0$. Below $0$, our surplus process $U$ evolves as $X=\{X_t=Y_t+\delta t, t\geq 0\}$. Clearly, $X$ is also a Lévy insurance risk process; in fact, $X$ and $Y$ share many properties except for those affected by the value of the linear part of the L\'evy process.

In other words, our surplus process is given by the solution $U=\{U_t, t\geq 0\}$ to the following stochastic differential equation: for $\delta \geq 0$,
\begin{equation}\label{E:dynamic}
\mathrm{d}U_t = \mathrm{d}Y_t + \delta \ind_{\{U_t < 0\}} \mathrm{d}t , \quad t \geq 0 .
\end{equation}

 
\subsection{Lévy insurance risk processes}

We say that $X=\{X_t,t\geq 0\}$ is a Lévy insurance risk process if it is a spectrally negative L\'evy process (SNLP) on the filtered probability space $(\Omega,\mathcal{F},\{\mathcal{F}_t , t\geq0\}, \mathbb{P})$, that is a process with stationary and independent increments and no positive jumps. To avoid trivialities, we exclude the case where $X$ has monotone paths.

As the L\'{e}vy process $X$ has no positive jumps, its Laplace exponent exists: for all $\lambda, t \geq 0$,
$$
\e \left[ \mathrm{e}^{\lambda X_t} \right] = \mathrm{e}^{t \psi(\lambda)} ,
$$
where
$$
\psi(\lambda) = \gamma \lambda + \frac{1}{2} \sigma^2 \lambda^2 + \int^{\infty}_0 \left( \mathrm{e}^{-\lambda z} - 1 + \lambda z \ind_{(0,1]}(z) \right) \Pi(\mathrm{d}z) ,
$$
for $\gamma \in \reals$ and $\sigma \geq 0$, and where $\Pi$ is a $\sigma$-finite measure on $(0,\infty)$ such that
$$
\int^{\infty}_0 (1 \wedge z^2) \Pi(\mathrm{d}z) < \infty .
$$
This measure $\Pi$ is called the L\'{e}vy measure of $X$. Finally, note that $\e \left[ X_1 \right] = \psi'(0+)$ and thus, in a Lévy insurance risk model, the \textit{net profition condition} is written $\e \left[ X_1 \right] = \psi'(0+) \geq 0$. We will use the standard Markovian notation: the law of $X$ when starting from $X_0 = x$ is denoted by $\p_x$ and the corresponding expectation by $\e_x$. We write $\p$ and $\e$ when $x=0$.

When the surplus process $X$ has paths of bounded variation, that is when $\int^{1}_0 z \Pi(\mathrm{d}z)<\infty$ and $\sigma=0$, we can write
$$
X_t = \drift t - S_t ,
$$
where $\drift := \gamma+\int^{1}_0 z \Pi(\mathrm{d}z) > 0$ is the drift of $X$ and where $S=\{S_t,t\geq 0\}$ is a driftless subordinator (e.g.\ a Gamma process or a compound Poisson process).

We now present the definition of the scale functions $W^{(q)}$ and $Z^{(q)}$ of $X$. First, recall that there exists a function $\Phi \colon [0,\infty) \to [0,\infty)$ defined by $\Phi(q) = \sup \{ \lambda \geq 0 \mid \psi(\lambda) = q\}$ (the right-inverse of $\psi$) such that
$$
\psi ( \Phi(q) ) = q, \quad q \geq 0 .
$$
Now, for $q \geq 0$, the $q$-scale function of the process $X$ is defined as the continuous function on $[0,\infty)$ with Laplace transform
\begin{equation}\label{def_scale}
\int_0^{\infty} \mathrm{e}^{- \lambda y} W^{(q)} (y) \mathrm{d}y = \frac{1}{\psi(\lambda) - q} , \quad \text{for $\lambda > \Phi(q)$.}
\end{equation}
This function is unique, positive and strictly increasing for $x\geq0$ and is further continuous for $q\geq0$. We extend $W^{(q)}$ to the whole real line by setting $W^{(q)}(x)=0$ for $x<0$. We write $W = W^{(0)}$ when $q=0$. We also define
\begin{equation}\label{eq:zqscale}
Z^{(q)}(x) = 1 + q \int_0^x W^{(q)}(y)\mathrm d y, \quad x \in \mathbb R.
\end{equation}

%

%

If we define $Y=\{Y_t=X_t-\delta t, t\geq 0\}$, then it is also a Lévy insurance risk process (if it doesn't have monotone paths): its linear part is given by $\gamma-\delta$ but it has the same Gaussian coefficient $\sigma$ and L\'evy measure $\Pi$ as $X$. In fact, $X$ and $Y$ share many properties. Note that we could have specified $Y$ first and then define $X=\{X_t=Y_t+\delta t, t\geq 0\}$ as in the Introduction. The two approaches are equivalent.

The Laplace exponent of $Y$ is given by
$$
\lambda \mapsto \psi(\lambda) - \delta \lambda ,
$$
with right-inverse $\varphi(q) = \sup \{ \lambda \geq 0 \mid \psi(\lambda) - \delta \lambda = q\}$. Then, for each $q \geq 0$, we define its scale functions $\mathbb{W}^{(q)}$ and $\mathbb{Z}^{(q)}$ as in Equations~\eqref{def_scale} and~\eqref{eq:zqscale}:
$$
\int_0^{\infty} \mathrm{e}^{- \lambda y} \mathbb{W}^{(q)} (y) \mathrm{d}y = \frac{1}{\psi(\lambda) - \delta \lambda - q} , \quad \text{for $\lambda > \varphi(q)$}
$$
and
$$
\mathbb{Z}^{(q)}(x) = 1 + q \int_0^x \mathbb{W}^{(q)}(y)\mathrm d y, \quad x \in \mathbb R.
$$


\subsection{Refracted Lévy processes}

Recall from Equation~\eqref{E:dynamic}, that our surplus process $U=\{U_t, t\geq 0\}$ is equivalently the solution to
$$
\mathrm{d}U_t = \mathrm{d}Y_t + \delta \ind_{\{U_t < 0\}} \mathrm{d}t , \quad t \geq 0 ,
$$
or
$$
\mathrm{d}U_t = \mathrm{d}X_t - \delta \ind_{\{U_t > 0\}} \mathrm{d}t , \quad t \geq 0 ,
$$
where $\delta \geq 0$ is the refraction parameter. The second stochastic differential equation is the one used in \cite{kyprianouloeffen2010}. It was proved in that article that such a process exists and that it is a skip-free upward strong Markov process.

For technical reasons, we need to assume that if $X$ (and also $Y$) has paths of bounded variation then
\begin{equation}\label{E:delta}
0 \leq \delta < \drift = \gamma + \int_{(0,1)} z \Pi(\mathrm{d}z) .
\end{equation}  
Since in this case, $X$ may be written as $X_t = \drift t - S_t$, the condition in Equation~\eqref{E:delta} amounts to making sure $Y$ has a strictly positive linear drift.

In \cite{kyprianouloeffen2010}, many fluctuation identities, including the probability of ruin for $U$, have been derived using \textit{scale functions} for $U$: for $q \geq 0$ and for $x \in \reals$, set
\begin{equation}\label{small w}
w^{\left(q\right)}(x;z) = W^{(q)}(x-z) + \delta \ind_{\{x \geq 0\}} \int^{x}_0 \WW^{(q)}(x-y) W^{(q)\prime}(y-z) \md y .
\end{equation}
Note that when $x<0$, we have
$$
w^{\left(q\right)}(x;z)=W^{\left(q\right)}(x-z) .
$$
For $q=0$, we will write $w^{(0)}(x;z)=w(x;z)$. See \cite{kyprianou2014} for more details.


\subsection{Classical ruin and exit problems}

Here is a collection of known fluctuation identities for the spectrally negative L\'evy processes $X$ and $Y$, as well as for the refracted Lévy process $U$. See \cite{kyprianou2014} for more details.

First, for real numbers $a$ and $b$, we define the following first-passage stopping times:
\begin{align*}
\tau_a^- &= \inf\{t>0 \colon X_t<a\} \quad \textrm{and} \quad \tau_b^+ = \inf\{t>0 \colon X_t\geq b\} \\
\nu_a^- &= \inf\{t>0 \colon Y_t<a\} \quad \textrm{and} \quad \nu_b^+ = \inf\{t>0 \colon Y_t\geq b\} \\
\kappa_a^- &= \inf\{t>0 \colon U_t<a\} \quad \textrm{and} \quad \kappa_b^+ = \inf\{t>0 \colon U_t\geq b\} ,
\end{align*}
with the convention $\inf \emptyset=\infty$. For $a \leq 0 \leq b$ and $q \geq 0$, if $a \leq x \leq b$ then we have
%
%
%
$$
\e_x \left[ \mathrm{e}^{-q \kappa_b^+} \ind_{\{\kappa_b^+< \kappa_a^-\}} \right] = \frac{\wq(x;a)}{\wq(b;a)} ,
$$
from which we can deduce that
\begin{equation}\label{eq:first-passage_above}
\e_x \left[ \mathrm{e}^{-q \kappa_b^+} \ind_{\{\kappa_b^+< \infty\}} \right] = \frac{\mathrm{e}^{\Phi(q)x} + \delta \Phi(q) \ind_{\{x \geq 0\}} \int_0^x \mathrm{e}^{\Phi(q)y} \mathbb{W}^{(q)}(x-y) \mathrm{d}y}{\mathrm{e}^{\Phi(q) b} + \delta \Phi(q) \int_0^b \mathrm{e}^{\Phi(q)y} \mathbb{W}^{(q)}(b-y) \mathrm{d}y} .
\end{equation}
See Theorem 5 in \cite{kyprianouloeffen2010}.

Moreover, the \textit{classical} probability of ruin, associated with each three processes, is given by
\begin{equation}\label{E:classicalruinprobaX}\nonumber
\p_x \left( \tau_0^- < \infty \right) = 1 - \left( \e \left[ X_1 \right]\right)_{+} W(x) ,
\end{equation}
for $X$, while for $Y$ and $U$ we have
\begin{equation}\label{E:classicalruinprobaY}
\p_x \left( \nu_0^- < \infty \right) = 1 - \left( \e \left[ X_1\right]-\delta \right)_{+} \WW(x)
\end{equation}
and
\begin{equation}\label{E:classicalruinprobaU}
\p_x \left( \kappa_0^- < \infty \right) = 1 -\frac{\left(\e \left[ X_1 \right] - \delta \right)_{+}}{1-\delta W(0)} w(x;0) .
\end{equation}
Of course, the expressions in Equations~\eqref{E:classicalruinprobaY} and~\eqref{E:classicalruinprobaU} should be equal because $\left\lbrace Y_t , t<\nu^-_{0} \right\rbrace$ and $\left\lbrace U_t , t<\kappa^-_0 \right\rbrace$ have the same distribution with respect to $\p_{x}$ when $x>0$. Using Equation~\eqref{E:convolution} from the Appendix, we can see that this is the case.

Finally, since the Laplace exponent of $Y$ is given by $\lambda \mapsto \psi(\lambda) - \delta \lambda$, then for $x, \theta>0$ we have
\begin{equation}\label{laplruinovershoot}
 \mathbb{E}_x \left[ \mathrm{e}^{\theta Y_{\nu_0^-}} \ind_{\{\nu_0^-<\infty\}} \right] = \mathrm{e}^{\theta x} -\left(\psi(\theta)-\delta \theta \right) \mathrm{e}^{\theta x}\int_0^x \mathrm{e}^{-\theta z} \mathbb{W}(z)\mathrm{d}z - \frac{\psi(\theta)-\delta \theta}{\theta} \mathbb{W}(x) .
\end{equation}

We conclude this section with definitions of auxiliary functions. For the sake of compactness, we define for $p,p+q\geq 0$ and $a,x \in \reals$
\begin{align}\label{eq:convsnlp}
\mathcal{W}_{a}^{\left( p,q\right) }\left( x\right) &= W^{\left( p\right)}\left( x\right) +q\int_{a}^{x}W^{\left( p+q\right) }\left( x-y\right)W^{\left( p\right) }\left( y\right) \md y \nonumber\\ &= W^{\left(p+q\right)}\left( x\right)-q\int_{0}^{a}W^{\left( p+q\right) }\left( x-y\right) W^{\left( p\right) }\left( y\right) \md y
\end{align}
and 
\begin{align}\label{eq:convsnlp2}
\mathcal{H}^{\left( p,q\right) }(x) =\me^{\Phi\left( p\right)x}\left(1+q\int_{0}^{x}\me^{-\Phi\left( p\right)z}W^{\left( p+q\right)}(y)\md y\right),
\end{align}
where the second equality in Equation~\eqref{eq:convsnlp} follows from identity~\eqref{eq:sym} in the Appendix. We also define 
\begin{align}\label{eq:convsnlpU}\nonumber
\mathcal{W}_{a,\delta}^{\left( p,q\right)} \left( x\right) &= \WW^{\left(p\right)}(x)-\delta W^{\left(p+q\right)}(0)\WW^{\left(p\right)}\left( x\right) \\
& +\int_{a}^{x}\left(qW^{\left( p+q\right)}\left(x-y\right)-\delta W^{\left( p+q\right)'}\left(x-y\right)\right)\WW^{\left( p\right)}\left(y\right) \md y\nonumber\\
&= W^{\left(p+q\right)}(x)-\int_{0}^{a}\left(qW^{\left( p+q\right)}\left(x-y\right)-\delta W^{\left( p+q\right)'}\left(x-y\right)\right)\WW^{\left( p\right)}\left(y\right) \md y,
\end{align}
and 
\begin{align*}
\mathcal{H}^{\left( p,q\right) }_{\delta}(x) =\me^{\varphi \left( p\right)x}
\left(1+(q-\delta \varphi \left( p\right))\int_{0}^{x}\me^{-\varphi \left( p\right)y}W^{\left( p+q\right)}(y)\md y\right),
\end{align*}
as analogues of~\eqref{eq:convsnlp} and~\eqref{eq:convsnlp2} respectively, where $\mathcal{H}^{\left( p,q\right) }_{0}=\mathcal{H}^{\left( p,q\right) }$ and $\mathcal{W}^{\left( p,q\right) }_{a,0}=\mathcal{W}^{\left( p,q\right)}_{a}$.

The second expression for $\mathcal{W}_{a,\delta}^{\left( p,q\right)}$ in Equation~\eqref{eq:convsnlpU} follows from identity~\eqref{E:convolution} in the Appendix.

\section{Main results}

Following the definition for a standard Lévy insurance risk process, we define the time of Parisian ruin, with delay $r>0$, for the refracted Lévy insurance risk process $U$ by
$$
\kappa_r^U = \inf \left\lbrace t > 0 \colon t - g_t^U > r \right\rbrace ,
$$
where $g_t^U = \sup \left\lbrace 0 \leq s \leq t \colon U_s \geq 0 \right\rbrace$. Our main objective is to obtain an expression for the corresponding probability of Parisian ruin that has a similar structure as the one in Equation~\eqref{SNLPPr}.

\begin{theorem}\label{refractedTh}
For $x\in \reals,$
\begin{equation}\label{refractedPr}
\p_{x} \left(\kappa^{U}_{r}<\infty \right) = 1 - \left(\e[X_{1}]-\delta\right)_{+} \frac{\int^{\infty}_{0}w(x;-z)z\p(X_{r}\in \md z)}{\int^{\infty}_{0}z\p(X_{r}\in \md z)-\delta r} .
\end{equation}
\end{theorem}

For classical ruin and Parisian ruin for a standard SNLP, if the \textit{net profit condition} is not verified then (Parisian) ruin occurs almost surely. In the last result, if $\e[X_{1}] \leq \delta$, then the probability of Parisian ruin for $U$ is equal to $1$. This is because asking for $\e[Y_{1}]=\e[X_{1}] - \delta > 0$ is the same as the net profit condition in this model, namely for the surplus process $U$.

Also, it should be clear that, if we set $\delta=0$ in the above result, then we recover Equation~\eqref{SNLPPr}.

\begin{remark}
Using identities from Section~\ref{S:proofs}, we can also re-write the result in Equation~\eqref{refractedPr} as follows:
$$
\p_{x} \left(\kappa^{U}_{r}<\infty \right) = 1 - \left(\e[X_{1}]-\delta\right)_{+} \frac{\int^{\infty}_{0} w(x;-z) z \p(X_{r}\in \md z)}{\int^{\infty}_{0}\left(1-\delta W\left(z\right)\right) z \p(X_{r}\in \md z)} .
$$
\end{remark}

\subsection{Other results}

Using some of the results/lemmas in Section~\ref{S:proofs}, it is possible to obtain other fluctuation identities for $U$ involving the time of Parisian ruin.

For example, the discounted probability of $U$ reaching level $a$ before being Parisian ruined and the Laplace transform of the time of Parisian ruin time can also be computed.
\begin{theorem}\label{refracted}
For any $x \leq a$ and $q\geq 0$, we have
\begin{enumerate}
\item[(i)] \begin{align*}
&\e_{x}\left[\mathrm{e}^{-q(\kappa^{U}_{r}-r)}\ind_{\left\lbrace \kappa^{U}_{r}<\kappa_{a}^{+}
\right\} }\right]\\ & 
\qquad=\mathbb{Z}^{\left(q\right)}\left(x\right)+\int_{0}^{\infty}\left( w^{\left(q\right)}\left(x;-z\right)\e\left[\me^{-q\kappa^{U}_{r}}\ind_{\left\lbrace \kappa^{U}_{r}<\kappa_a^{+}\right\rbrace}\right]
-\mathcal{W}_{x,\delta}^{\left( q,-q\right) }\left(x+z\right)
 \right)\frac{z}{r}\p\left( X_{r}\in \md z\right),
\end{align*}
where 
\begin{align*}
&\e\left[ \me^{-q\kappa _{r}^{U}}\ind_{\left\lbrace \kappa_{r}^{U}<\kappa _{a}^{+} \right\rbrace }\right]=1-\frac{\mathbb{Z}^{\left(
q\right)}\left(a\right)
+\int_{0}^{\infty }\left(w^{\left(q\right) }\left(a;-z\right)-\mathcal{W}_{a,\delta}^{\left(q,-q\right) }\left(a+z\right) \right)\frac{z}{r}\p\left( X_{r}\in \md z\right)
}
{
\int_{0}^{\infty }w^{\left(q\right) }\left(a;-z\right)\frac{z}{r}\p\left( X_{r}\in \md z\right)
}
\\
& \qquad=\frac{\int_{0}^{\infty }\mathcal{W}_{a,\delta}^{\left(q,-q\right) }\left(a+z\right) \frac{z}{r}\p\left( X_{r}\in \md z\right)}{
\int_{0}^{\infty }w^{\left(q\right) }\left(a;-z\right)\frac{z}{r}\p\left( X_{r}\in \md z\right)
}
 -\frac{\mathbb{Z}^{\left(
q\right)}\left(a\right) 
}
{
\int_{0}^{\infty }w^{\left(q\right) }\left(a;-z\right)\frac{z}{r}\p\left( X_{r}\in \md z\right)
}
\end{align*}

\item[(ii)]
\begin{align*}
&\e_{x}\left[\mathrm{e}^{-q(\kappa^{U}_{r}-r)}\ind_{\left\lbrace \kappa^{U}_{r}<\infty
\right\} }\right]\\ & 
\qquad=\mathbb{Z}^{\left(q\right)}\left(x\right)+\int_{0}^{\infty}\left( w^{\left(q\right)}\left(x;-z\right)\e\left[\me^{-q\kappa^{U}_{r}}\ind_{\left\lbrace \kappa^{U}_{r}<\infty\right\rbrace}\right]
-\mathcal{W}_{x,\delta}^{\left( q,-q\right) }\left(x+z\right)
 \right)\frac{z}{r}\p\left( X_{r}\in \md z\right),
\end{align*}
where 
$$
\e \left[\mathrm{e}^{-q(\kappa^{U}_{r}-r)}\ind_{\left\lbrace \kappa^{U}_{r}<\infty\right\} }\right] = \frac{\int_{0}^{\infty}\mathcal{H}^{\left(q,-q\right)}_{\delta}(z)\frac{z}{r}\p\left( X_{r}\in \md z\right)-\frac{q}{\varphi \left( q\right)}-\delta}{\int_{0}^{\infty }\mathcal{H}^{\left(q,0\right)}_{\delta}(z)\frac{z}{r}\p\left( X_{r}\in \md z\right)-\delta \me^{qr}} ,
$$
\item[(iii)]
\begin{align*}
&\e_{x}\left[\mathrm{e}^{-q\kappa^{+}_{a}}\ind_{\left\lbrace \kappa_a^+<\kappa^{U}_{r}
\right\} }\right]=\frac{\int_{0}^{\infty} w^{\left(q\right)}(x;-z) \frac{z}{r}\p
\left( X_{r}\in \md z\right)}{\int_{0}^{\infty} w^{\left(q\right)}(a;-z) \frac{z}{r}\p
\left( X_{r}\in \md z\right)}.
\end{align*}
\end{enumerate}
\end{theorem}
\begin{remark}
If we set $\delta=0$, we obtain the same quantities by replacing $\varphi$, $w^{(q)}$, $\mathcal{H}^{\left(q,-q\right)}_{\delta}$ and $\mathcal{W}^{\left(q,-q\right)}_{\delta}$  by $\Phi$, $W^{(q)}$, $\mathcal{H}^{\left(q,-q\right)}$ and $\mathcal{W}^{\left(q,-q\right)}$ respectively.
\end{remark}
\section{Examples}

We now present four models in which we can compute the probability of Parisian ruin given in Theorem~\ref{refractedTh}. The task amounts to finding processes $X$ and $Y$ for which both the distribution and the scale function are known. First, we will look at the two classical models: the Cramér-Lundberg model with exponential claims and the Brownian risk model. Then, we will move toward more sophisticated surplus processes, namely a stable risk process and a jump-diffusion risk process with phase-type claims.


\subsection{Cramér-Lundberg processes with exponential claims}

When $X$ and $Y$ are a Cramér-Lundberg risk processes with exponentially distributed claims, then they are given by
$$
X_t - X_0 = \drift t - \sum_{i=1}^{N_t} C_i \quad \text{and} \quad Y_t - Y_0 = (\drift-\delta) t - \sum_{i=1}^{N_t} C_i ,
$$
where $N=\{N_t, t\geq 0\}$ is a Poisson process with intensity $\eta>0$, and where $\{C_1, C_2, \dots\}$ are independent and exponentially distributed random variables with parameter $\alpha$. The Poisson process and the random variables are mutually independent. In this case, the Laplace exponent of $X$ is given by
$$
\psi(\lambda) = \drift \lambda + \eta \left( \frac{\alpha}{\lambda+\alpha} - 1 \right) , \quad \text{for $\lambda>-\alpha$}
$$
and the net profit condition is given by $\e \left[ Y_1 \right] = \drift - \delta - \eta/\alpha \geq 0$. Then, for $x\geq 0$, we have
\begin{align*}
W(x) &= \frac{1}{c-\eta/\alpha} \left( 1- \frac{\eta}{c\alpha}\mathrm{e}^{(\frac{\eta}{c}-\alpha)x} \right), \\
\mathbb{W}(x) &= \frac{1}{c-\delta-\eta/\alpha} \left( 1- \frac{\eta}{(c-\delta)\alpha}\mathrm{e}^{\left(\frac{\eta}{c-\delta}-\alpha\right)x} \right), \\
w\left(x;-z\right) &=\frac{1}{c-\eta/\alpha} \left( 1- \frac{\eta}{c\alpha}\mathrm{e}^{(\frac{\eta}{c}-\alpha)(x+z)} \right) + \frac{K(x,\delta,\alpha,\eta,c)}{(c-\delta-\eta/\alpha)c}\mathrm{e}^{(\frac{\eta}{c}-\alpha)z} ,
\end{align*}
where
$$
K(x,\delta,\alpha,\eta,c) :=\delta \eta\left(\frac{1}{\eta-c \alpha}\left(\mathrm{e}^{(\frac{\eta}{c}-\alpha)x} -1\right)-\frac{1}{\delta \alpha}\left(1-\mathrm{e}^{\frac{-\eta \delta}{c(c-\delta)}x} \right)\mathrm{e}^{(\frac{\eta}{c-\delta}-\alpha)x} \right) .
$$

As noted in \cite{loeffenetal2013}, we have
\begin{multline*}
\mathbb{P} \left( \sum_{i=1}^{N_r}C_i\in\mathrm{d}y \right) = \sum_{k=0}^\infty \mathbb P \left( \sum_{i=0}^k C_i\in\mathrm{d}y \right) \mathbb P(N_r=k)  \\
= \mathrm{e}^{-\eta r} \left( \delta_0(\mathrm{d}y)  + \mathrm{e}^{-\alpha y}\sum_{m=0}^\infty \frac{ (\alpha \eta r)^{m+1}}{m!(m+1)!}  y^{m} \mathrm{d}y  \right) ,
\end{multline*}
where $\delta_0(\mathrm{d}y)$ is a Dirac mass at $0$, and consequently
\begin{align*}
\int_0^\infty   z\mathbb{P}(X_r\in\mathrm{d}z) &= \int_0^{cr} z\mathrm{e}^{-\eta r} \left( \delta_0(cr-\mathrm{d}z)  + \mathrm{e}^{-\alpha (cr-z)}\sum_{m=0}^\infty \frac{ (\alpha \eta r)^{m+1}}{m!(m+1)!}  (cr-z)^{m} \mathrm{d}z  \right) \\
&= \mathrm{e}^{-\eta r} \left(  cr + \sum_{m=0}^\infty \frac{ (\eta r)^{m+1}}{m!(m+1)!} \left[ cr\Gamma(m+1,cr\alpha) -\frac1\alpha \Gamma(m+2,cr\alpha)  \right]  \right) ,
\end{align*}
where $\Gamma(a,x)=\int_0^x \mathrm{e}^{-t}t^{a-1}\mathrm{d}t$ is the incomplete gamma function, and
\begin{equation*}
 \frac{\eta}{c\alpha}\int_0^\infty \mathrm{e}^{(\frac{\eta}{c}-\alpha)z}z\mathbb{P}(X_r \in \mathrm{d}z) = \int_0^\infty z \mathbb{P}(X_r\in\mathrm{d}z)- (c-\eta/\alpha )r .
\end{equation*}

Putting all the pieces together with the main result of Theorem~\ref{refractedTh}, we obtain the following expression for the probability of Parisian ruin:
\begin{multline*}
\mathbb{P}_x(\kappa^{U}_r<\infty) =1-\left(1-\frac{\delta}{c-\eta/\alpha}\right)\left(1-\mathrm{e}^{(\frac{\eta}{c}-\alpha)x}\right)\\
 -\left(1-\frac{\delta}{c-\eta/\alpha}\right)\frac{\delta r-\mathrm{e}^{(\frac{\eta}{c}-\alpha)x}\left( \delta r -(c-\eta/\alpha)r\right)}{\mathrm{e}^{-\eta r}\left(cr + \sum_{m=0}^\infty \frac{ (\eta r)^{m+1}}{m!(m+1)!} \left[ cr\Gamma(m+1,cr\alpha) -\frac1\alpha \Gamma(m+2,cr\alpha)  \right]\right)-\delta r} \\
- \frac{ \alpha}{\eta}K(x,\delta,\alpha,\eta,c)\left(1+\frac{\delta r-(c-\eta/\alpha)r}{\mathrm{e}^{-\eta r}\left(cr + \sum_{m=0}^\infty \frac{ (\eta r)^{m+1}}{m!(m+1)!} \left[ cr\Gamma(m+1,cr\alpha) -\frac1\alpha \Gamma(m+2,cr\alpha)  \right]\right)-\delta r}\right) .
\end{multline*}

The following two tables provide a sensitivity analysis for the probability of Parisian ruin in a refracted Cram\'er-Lundberg model (with exponential claims) with respect to the refraction parameter $\delta$ and the Parisian delay parameter $r$. The value of the initial level $U_0=x$ is also varying.

Note that, in this example, we used the notation $\drift$ for the linear part of $X$ (below $0$) and $\drift-\delta$ for the linear part of $Y$ (above $0$). In other words, during \textit{regular business periods}, the drift is given by $\drift-\delta$. Consequently, in Table~\ref{table:table1}, we have fixed the value of $\drift-\delta$ (above $0$) and looked at the effect of a change in value of $\delta$, the refraction parameter, on the probability of Parisian ruin. Note that, when $\delta$ increases, then the value of $\drift$ (below $0$) also increases to keep $c-\delta$ constant. As expected, the larger the value of $\delta$, the smaller the probability of Parisian ruin.

In Table~\ref{table:table2}, we have fixed all parameters except for the Parisian delay parameter $r$. As expected, the larger the value of the delay $r$, i.e.\ the larger the \textit{grace period}, the smaller the probability of Parisian ruin.

\begin{table}[h]
\caption{Impact of the refraction parameter $\delta$ on the probability of Parisian ruin in a refracted Cram\'er-Lundberg model}
\centering
\begin{tabular}{|c|c|c|c|c|}
\hline
$x$ & $\delta =0$ & $\delta =1$ & $\delta =3$ & $\delta =5$ \\ \hline\hline
$1$ & $2.872324151\times 10^{-1}$ & $1.850876547\times 10^{-1}$ & $%
5.573334777\times 10^{-2}$ & $1.226635655\times 10^{-2}$ \\ \hline
$5$ & $1.474700390\times 10^{-1}$ & $9.50271705\times 10^{-2}$ & $%
2.86144548\times 10^{-2}$ & $6.2977571\times 10^{-3}$ \\ \hline
$10$ & $6.40902148\times 10^{-2}$ & $4.12986379\times 10^{-2}$ & $%
1.24357907\times 10^{-2}$ & $2.7369940\times 10^{-3}$ \\ \hline
$20$ & $1.210507796\times 10^{-2}$ & $7.8003051\times 10^{-3}$ & $%
2.3488176\times 10^{-3}$ & $5.169513\times 10^{-4}$ \\ \hline
$30$ & $2.286353896\times 10^{-3}$ & $1.4732872\times 10^{-3}$ & $%
4.436344\times 10^{-4}$ & $9.76391\times 10^{-6}$ \\ \hline
\end{tabular}
\begin{minipage}{12cm}
Parameters: $r=2$, $\drift-\delta=6$ (drift above $0$), $\eta=5$, $\alpha=1$.
\end{minipage}
\label{table:table1}
\end{table}

\begin{table}[h]
\caption{Impact of the delay parameter $r$ on the probability of Parisian ruin in a refracted Cram\'er-Lundberg model}
\centering
\begin{tabular}{|c|c|c|c|c|}
\hline
$x$ & $r=0$ & $r=1$ & $r=2$ & $r=3$ \\ \hline\hline
$1$ & $7.054014374\times 10^{-1}$ & $1.727546072\times 10^{-1}$ & $%
5.573334777\times 10^{-2}$ & $2.064556230\times 10^{-2}$ \\ \hline
$5$ & $3.621651737\times 10^{-1}$ & $8.86951728\times 10^{-2}$ & $%
2.86144548\times 10^{-2}$ & $1.05997853\times 10^{-2}$ \\ \hline
$10$ & $1.573963357\times 10^{-1}$ & $3.85467632\times 10^{-2}$ & $%
1.24357907\times 10^{-2}$ & $4.6066476\times 10^{-3}$ \\ \hline
$20$ & $2.972832780\times 10^{-2}$ & $7.2805432\times 10^{-3}$ & $%
2.3488176\times 10^{-3}$ & $8.700832\times 10^{-4}$ \\ \hline
$30$ & $5.614955832\times 10^{-3}$ & $1.3751168\times 10^{-3}$ & $%
4.436344\times 10^{-4}$ & $1.643375\times 10^{-4}$ \\ \hline
\end{tabular}
\begin{minipage}{12cm}
Parameters: $\delta=3$, $\drift=6$ (drift below $0$), $\drift-\delta=6$ (drift above 0), $\eta=5$, $\alpha=1$.
\end{minipage}
\label{table:table2}
\end{table}

\subsection{Brownian risk processes}

Now, if $X$ and $Y$ are Brownian risk processes, i.e.\ if
$$
X_t - X_0 = \drift t +\sigma B_t \quad \text{and} \quad Y_t -Y_0 = (\drift-\delta) t + \sigma B_t ,
$$
where $B=\{B_t, t\geq 0\}$ is a standard Brownian motion. In this case, the Laplace exponent of $X$ is given by
$$
\psi(\lambda) = \drift \lambda + \frac12\sigma^2 \lambda^2
$$
and the net profit condition is given by $\e \left[ Y_1 \right] = \drift - \delta \geq 0$. Then, for $x\geq 0$, we have
\begin{align*}
W(x) &= \frac{1}{c} \left( 1-\mathrm{e}^{-2\frac{c}{\sigma^2} x} \right), \\
\mathbb{W}(x) &= \frac{1}{c-\delta} \left( 1-\mathrm{e}^{-2\frac{c-\delta}{\sigma^2} x}  \right), \\
w\left(x;-z\right) &= \frac{1}{c} \left( 1-\mathrm{e}^{-2\frac{c}{\sigma^2} (x+z)}  \right) + M(x,\delta,\sigma,c)\mathrm{e}^{-2\frac{c}{\sigma^2} z} ,
\end{align*}
where
$$
M(x,\delta,\sigma,c) :=\frac{\delta}{c-\delta}\left(\frac{1}{c}\left(1-\mathrm{e}^{-2\frac{c}{\sigma^2}x}\right)-\frac{1}{\delta}\left(\mathrm{e}^{-2\frac{c-\delta}{\sigma^2}x}-\mathrm{e}^{-2\frac{c}{\sigma^2}x}\right)\right) .
$$

Again, as noted in \cite{loeffenetal2013}, we have
\begin{equation*}
 \int_0^\infty \mathrm{e}^{-\frac{2c}{\sigma^2} z}  z\mathbb{P}(X_r\in\mathrm{d}z) = \int_0^\infty   z\mathbb{P}(X_r\in\mathrm{d}z)- cr
\end{equation*}
and consequently
$$
\int_0^\infty z \mathbb{P}(X_r\in\mathrm{d}z) = \frac{1}{\sqrt{2\pi \sigma^2 r}}\int_0^\infty z\mathrm{e}^{-\frac{(z-c r)^2}{2\sigma^2 r}}\mathrm{d}z = \frac{\sigma\sqrt{r}}{\sqrt{2\pi}}\mathrm{e}^{-\frac{c^2r}{2\sigma^2}} +  cr \mathcal N \left( \frac{c\sqrt r}{\sigma}  \right) .
$$

Putting all the pieces together with the main result of Theorem~\ref{refractedTh}, we obtain the following expression for the probability of Parisian ruin:
\begin{multline*}
\mathbb{P}_x (\kappa^{U}_r<\infty) \\
= 1-\left(\frac{c-\delta}{c}\right)\frac{ \left(\frac{\sigma\sqrt{r}}{\sqrt{2\pi}}\mathrm{e}^{-\frac{c^2r}{2\sigma^2}} +  c r\mathcal N \left( \frac{c\sqrt r}{\sigma} \right)\right)\left(1- \mathrm{e}^{-\frac{2c}{\sigma^2} x} + cM(x,\delta,\sigma,c)\right)}{ \frac{\sigma\sqrt{r}}{\sqrt{2\pi}}\mathrm{e}^{-\frac{c^2r}{2\sigma^2}} +  c r\mathcal N \left( \frac{c\sqrt r}{\sigma}\right)-\delta r}\\
+\left(c-\delta \right)\frac{r\left(\mathrm{e}^{-\frac{2c}{\sigma^2} x}-cM(x,\delta,\sigma,c)\right)}{ \frac{\sigma\sqrt{r}}{\sqrt{2\pi}}\mathrm{e}^{-\frac{c^2r}{2\sigma^2}} +  c r\mathcal N \left( \frac{c\sqrt r}{\sigma}\right)-\delta r}.
\end{multline*}

The following two tables provide a sensitivity analysis for the probability of Parisian ruin in a refracted Brownian risk model with respect to the refraction parameter $\delta$ and the Parisian delay parameter $r$. The value of the initial level $U_0=x$ is also varying. Again in this example we used the notation $\drift$ for the linear part of $X$ (below $0$) and $\drift-\delta$ for the linear part of $Y$ (above $0$).

In Table~\ref{table:refraction}, we have fixed the value of $\drift-\delta$ (above $0$) and looked at the effect of a change in value of $\delta$, the refraction parameter, on the probability of Parisian ruin. As expected, the larger the value of $\delta$, the smaller the probability of Parisian ruin. In Table~\ref{table:delay}, we have fixed all parameters except for the Parisian delay parameter $r$. As expected, the larger the value of the delay $r$, i.e.\ the larger the \textit{grace period}, the smaller the probability of Parisian ruin.

\begin{table}[h]
\begin{center}
\caption{Impact of the refraction parameter $\delta$ on the probability of Parisian ruin in a refracted Brownian risk model}
\begin{tabular}{|c|c|c|c|c|c|}
\hline
$x$ & $\delta =0$ & $\delta =1$ & $\delta =3$ & $\delta =4$ & $\delta =5$ \\ 
\hline\hline
$1$ & $1.756316\times 10^{-2}$ & $4.058863\times 10^{-2}$ & $2.040134\times
10^{-2}$ & $1.393016\times 10^{-2}$ & $9.279776\times 10^{-3}$ \\ \hline
$5$ & $4.629599\times 10^{-3}$ & $1.069916\times 10^{-3}$ & $5.377735\times
10^{-2}$ & $3.671950\times 10^{-3}$ & $2.446123\times 10^{-3}$ \\ \hline
$10$ & $8.744183\times 10^{-4}$ & $2.020791\times 10^{-3}$ & $1.015725\times
10^{-3}$ & $6.935426\times 10^{-4}$ & $4.620132\times 10^{-4}$ \\ \hline
$20$ & $3.119399\times 10^{-5}$ & $7.209243\times 10^{-5}$ & $3.623682\times
10^{-4}$ & $2.474236\times 10^{-5}$ & $1.648221\times 10^{-5}$ \\ \hline
$30$ & $1.112814\times 10^{-6}$ & $2.574575\times 10^{-6}$ & $1.294587\times
10^{-6}$ & $8.835856\times 10^{-7}$ & $5.883359\times 10^{-7}$ \\ \hline
\end{tabular}
\label{table:refraction}
\begin{minipage}{12cm}
Parameters: $r=2$, $\drift-\delta=6$ (drift above $0$), $\sigma=6$
\end{minipage}
\end{center}
\end{table}

\begin{table}[h]
\begin{center}
\caption{Impact of the delay parameter $r$ on the probability of Parisian ruin in a refracted Brownian risk model}
\begin{tabular}{|c|c|c|c|c|c|}
\hline
$x$ & $r=0$ & $r=1$ & $r=2$ & $r=4$ & $r=6$ \\ \hline\hline
$1$ & $8.3650684\times 10^{-1}$ & $8.89538704\times 10^{-2}$ & $%
2.908344\times 10^{-2}$ & $5.066851\times 10^{-3}$ & $1.146373\times 10^{-3}$
\\ \hline
$5$ & $3.6513221\times 10^{-1}$ & $3.65700339\times 10^{-2}$ & $%
1.195692\times 10^{-2}$ & $2.083045\times 10^{-3}$ & $4.712679\times 10^{-4}$
\\ \hline
$10$ & $1.2674282\times 10^{-1}$ & $1.20385972\times 10^{-2}$ & $%
3.936133\times 10^{-3}$ & $6.857238\times 10^{-4}$ & $1.551377\times 10^{-4}$
\\ \hline
$20$ & $1.908693\times 10^{-2}$ & $1.3045990\times 10^{-3}$ & $%
4.265510\times 10^{-4}$ & $7.431054\times 10^{-5}$ & $1.681198~\times 10^{-5}
$ \\ \hline
$30$ & $3.41422\times 10^{-3}$ & $1.413768\times 10^{-4}$ & $4.622456\times
10^{-5}$ & $8.052897\times 10^{-6}$ & $1.821894\times 10^{-6}$ \\ \hline
\end{tabular}
\label{table:delay}
\begin{minipage}{12cm}
Parameters: $\delta=3$, $\drift=6$ (drift below $0$), $\drift-\delta=3$ (drift above $0$), $\sigma=6$
\end{minipage}
\end{center}
\end{table}

\subsection{Jump-diffusion risk processes with phase-type claims}

More generally, if we add a Brownian component and if we let the claim distribution be more general, then we consider a L\'evy jump-diffusion risk process with phase-type claims:
$$
X_t - X_0 = \drift t + \sigma B_t - \sum_{i=1}^{N_t} C_i \quad \text{and} \quad Y_t - Y_0 = (\drift-\delta) t + \sigma B_t - \sum_{i=1}^{N_t} C_i ,
$$
where $\sigma \geq 0$, $B=\{B_t, t\geq 0\}$ is a standard Brownian motion, $N=\{N_t, t\geq 0\}$ is a Poisson process with intensity $\eta>0$, and where $\{C_1, C_2, \dots\}$ are independent random variables with common phase-type distribution with with the minimal representation $(m, \mathbf{T}, \boldsymbol\alpha)$, i.e.\ its cumulative distribution function (cdf) is given by $F(x)=1-\boldsymbol\alpha \mathrm{e}^{\mathbf{T}x} \mathbf{1}$ and $\mathbf{T}$ is an $m\times m$ matrix of a continuous-time killed Markov chain, its initial distribution is given by a simplex
$\boldsymbol\alpha =[\alpha_1,...,\alpha_m]$ and $\mathbf{1}$ denotes a column vector of ones. All of the aforementioned objects are mutually independent (for details we refer to \cite{egamiyamazaki2014}).

The Laplace exponent of $X$ is then clearly given by
\begin{equation}\label{psi_Jump-diffusion}
\psi(\lambda) = c \lambda + \frac{\sigma^2 \lambda^2}{2}+\eta \left(\boldsymbol\alpha(\lambda \mathbf{I}-\mathbf{T})^{-1}\mathbf{t}-1\right) ,
\end{equation}
where $\mathbf{t}=-\mathbf{T}\mathbf{1}$.

Let us denote by $\rho_j$ and $\zeta_i$ the roots with negative real parts of equations $\lambda \mapsto \psi(\lambda)=0$ and $\lambda \mapsto \psi(\lambda)-\delta \lambda=0$, respectively. Since we assume the net profit condition $\e[X_1]>\delta$, from Proposition 5.4 in \cite{kuznetsovetal2012}, we have that the $\rho_j$'s and the $\zeta_i$'s are distinct roots. Then, from Proposition 2.1 in \cite{egamiyamazaki2014} and Proposition 5.4 in \cite{kuznetsovetal2012}, we can obtain
\begin{align*}
W(x) &= \frac{1}{\psi'(0)}+\sum_{j \in \mathcal{I}_{\rho}}A_j \mathrm{e}^{\rho_j x}, \quad W'(x)=\sum_{j \in \mathcal{I}_{\rho}} \rho_j A_j \mathrm{e}^{\rho_j x}, \\
\mathbb{W}(x) &= \frac{1}{\psi'(0)-\delta}+\sum_{i \in \mathcal{I}_{\zeta}}B_i \mathrm{e}^{\zeta_i x},\\
w\left(x;-z\right) &= \frac{1}{\psi'(0)}+\sum_{j \in \mathcal{I}_{\rho}}A_j \mathrm{e}^{\rho_j (x+z)} \\
& \quad + \frac{1}{\psi'(0)-\delta} \sum_{j \in \mathcal{I}_{\rho}}\rho_j A_j \left(\mathrm{e}^{\rho_j x}-1\right) \mathrm{e}^{\rho_j z}+\sum_{j \in \mathcal{I}_{\rho}}\sum_{i \in \mathcal{I}_{\zeta}}\frac{\mathrm{e}^{\rho_j x}-\mathrm{e}^{\zeta_i x}}{\rho_j-\zeta_i}A_j B_i \mathrm{e}^{\rho_j z},
\end{align*}
where $A_j =\frac{1}{\psi'(\rho_j)}$ and $B_i =\frac{1}{\psi'(\zeta_i)-\delta}$, and where $\mathcal{I}_{\rho}$ and $\mathcal{I}_{\zeta}$ are the sets of indices corresponding to the $\rho_j$'s and the $\zeta_i$'s, respectively. Moreover, one can observe that the Laplace exponent in~\eqref{psi_Jump-diffusion} and $\psi(\lambda)-\delta \lambda$ are a ratio of two polynomials of degree $m+2$ and $m$ respectively. This is true of course if $\sigma>0$ and $c>0$.  On the other hand if $\sigma =0$ and $c>0$ we obtain ratio of two polynomials of degree $m+1 $ and $m$ respectively. Thus if we take $\psi(\lambda)=0$ and $\psi(\lambda)-\delta \lambda=0$ we will have $m+2$ or $m+1$ roots depending on whether $\sigma >0$ or $\sigma=0$. From \cite{kuznetsovetal2012}[Prop. 5.4 (ii)] we know that there are $m+1$ (or $m$) roots with negative real part. Hence $\mathbf{card} \left( \mathcal{I}_{\zeta} \right) = \mathbf{card} \left( \mathcal{I}_{\rho} \right) = m+1$ (or $\mathbf{card} \left( \mathcal{I}_{\zeta} \right) = \mathbf{card} \left(\mathcal{I}_{\rho} \right) = m$ if $\sigma=0$). 

Moreover,
\begin{eqnarray}\label{phase density}\nonumber
\mathbb{P}(X_{r}\in\mathrm{d}z) = \mathrm{e}^{-\eta r} \sum_{k=0}^\infty \frac{(\eta r)^k}{k!} \int_0^\infty F^{*k}(\mathrm{d}y)
\mathcal{N}\left((\mathrm{d}z+y-c r)\sigma\sqrt{r}\right) ,\label{density}
\end{eqnarray}
where $\mathcal{N}$ is the cdf of a standard normal random variable, $F^{*k}$ is the $k$-th convolution of $F$ and for $k=0$ we understand $F^{*0}(\mathrm{d}y)=\delta_0(\mathrm{d}y)$ to be a Dirac mass at $0$.

Putting all the pieces together, we obtain an expression for the probability of Parisian ruin.
\subsection{Stable risk processes}

Now, if $X$ and $Y$ are $3/2$-stable risk processes, i.e.\ if
$$
X_t - X_0 = \drift t + Z_t \quad \text{and} \quad Y_t -Y_0 = (\drift-\delta) t + Z_t ,
$$
where $Z=\{Z_t, t\geq 0\}$ is a spectrally negative $\alpha$-stable process with $\alpha=3/2$. In this case, the Laplace exponent of $X$ is given by $\psi(\lambda) = \drift \lambda + \lambda^{3/2}$. Then, for $x\geq 0$, we have
\begin{align*}
W(x) &= \frac{1-E_{1/2}(-c\sqrt{x})}{c}, \\
\mathbb{W}(x) &= \frac{1-E_{1/2}(-(c-\delta)\sqrt{x})}{c-\delta}, \\
w\left(x;-z\right) &= \frac1c\left[1-E_{1/2} \left( -c\sqrt{x+z} \right) \right] \\
& \quad +\int_0^x \frac{1}{c-\delta}\left[1-E_{1/2} \left( -(c-\delta)\sqrt{x-y} \right) \right]\left(\frac{1}{\pi \sqrt{x}}-c\cdot E_{1/2} \left( -c\sqrt{y+z} \right) \right) \mathrm{d}y ,
\end{align*}
where $E_{1/2}$ is the Mittag-Leffler function of order $1/2$.

Again, as noted in \cite{loeffenetal2013}, we have
\begin{equation}\label{stable density}\nonumber
\mathbb P(Z_r\in\mathrm{d}y)= \mathbb P(r^{2/3}Z_1\in\mathrm{d}y) =
\begin{cases}
\sqrt{\frac3{\pi}} r^{2/3} y^{-1} \mathrm{e}^{-u/2}W_{1/2,1/6}\left(u\right)\mathrm{d}y & y>0, \\
 -\frac{1}{2\sqrt{3\pi}}r^{2/3} y^{-1} \mathrm{e}^{u/2}W_{-1/2,1/6}\left(u\right)\mathrm{d}y & y<0,
\end{cases}
\end{equation}
where $u=\frac{4}{27}r^{9/2}|y|^3$ and $W_{\kappa,\mu}$ is Whittaker's W-function (not to be confused with the $0$-scale function of $X$).

Putting all the pieces together with the main result of Theorem~\ref{refractedTh}, we obtain the probability of Parisian ruin.

\section{Proofs and more}\label{S:proofs}

The proofs of our main results are based on technical but important lemmas (provided in the next section), as well as more standard probabilistic decompositions.
%

\subsection{Intermediate results}

The next lemma is lifted from \cite{loeffenetal2013}:
\begin{lemma}\label{laplace}
For $\theta>q>0$ and $y\geq0$,
\begin{equation}\label{eq:lemmapart1}
\int_0^\infty \mathrm{e}^{-\theta r} \int_{y}^\infty \frac{z}{r} \mathbb P(X_r\in\mathrm{d}z)\mathrm{d}r = \frac{1}{\Phi(\theta)}\mathrm{e}^{-\Phi(\theta)y} ,
\end{equation}
and
\begin{equation}\label{eq:lemmapart2}
\int_0^\infty \mathrm{e}^{-\theta r} \int^{\infty}_{0}W^{(q)}(z-y)\frac{z}{r}\p(X_{r}\in \md z) \md r = \frac{\mathrm{e}^{-\Phi(\theta)y }}{\theta-q} .
\end{equation}
\end{lemma}

From this first lemma, we can deduce the following two useful identities:
\begin{equation}
\int_{0}^{\infty }W^{(q)}(z)\frac{z}{r}\mathbb{P}(X_{r}\in \mathrm{d}z)=\mathrm{e}^{qr},  \label{eq:lemmapart3}
\end{equation}
and 
\begin{equation}\label{eq:lemmapart4}
\int_0^\infty \mathrm{e}^{-\theta r}W(z-y)\frac{z}{r}\mathbb{P}(X_{r}\in \mathrm{d}z) = \frac{1}{\theta} \mathrm{e}^{-\Phi(\theta)y }, \quad y\geq0 .
\end{equation}

We can also extract from \cite{loeffenetal2013} the following identity: for $x<0$,
\begin{equation}\label{eq:lemmapart5}
\p_x \left(\tau_0^+\leq r\right) = \int^{\infty}_{0}W(x+z)\frac{z}{r}\p(X_{r}\in\md z) .
\end{equation}
This identity will be generalized in Equation~\eqref{E:L3}.

For the proof of our main lemma, which is Lemma~\ref{L:main_lemma} below, we will need the following result taken from \cite{loeffen2014}.
\begin{lemma}
For all $p,q \geq 0$ and $a\leq x\leq b$,
\begin{multline}\label{eq:exp_scale}
\e_x \left[ \mathrm{e}^{-p\nu_a^-} W^{(q)}(Y_{\nu_a^-})\ind_{\{\nu_a^-<\nu_b^+\}}\right] = W^{(q)}(x)-\int_0^{x-a} \left((q-p) W^{(q)}(x-z)-\delta W^{(q)'}(x-z)\right)\WW^{(p)}(z)\md z \\
-\frac{\WW^{(p)}(x-a)}{\WW^{(p)}(b-a)} \left(W^{(q)}(b)-\int_0^{b-a} \left((q-p) W^{(q)}(b-z)-\delta W^{(q)'}(b-z)\right)\WW^{(p)}(z)\md z \right) .
\end{multline}
\end{lemma}
Note that another expression for~\eqref{eq:exp_scale} can be found in \cite[Lemma1]{renaud2014}.

The following three identities are new and crucial for the proofs of our main results.
\begin{lemma}\label{L:main_lemma}
For $x\in \reals$, $q \geq 0$ and $a\geq 0$, we have
\begin{multline}\label{E:L3}
\e_{x}\left[\p_{Y_{\nu^{-}_{0}}}\left(\tau^{+}_{0}\leq r\right)\ind_{\left\lbrace \nu^{-}_{0}<\infty\right\rbrace}\right]=\int^{\infty}_{0}\left(w(x;-z)-\WW(x)\right)\frac{z}{r}\p(X_{r}\in \md z)+\delta \WW(x) ,
\end{multline}
\begin{multline}\label{E:L1}
\e_{x}\left[ \me^{-q\nu _{0}^{-}}\e_{Y_{\nu_{0}^{-}}}\left[ \me^{-q\tau _{0}^{+}}\ind_{\left\{\tau _{0}^{+}\leq r\right\}
}\right]\ind_{\left\{\nu _{0}^{-}<\nu _{a}^{+} \right\} }\right]\\ =\int_{0}^{\infty }\me^{-qr}\left(w^{\left(q\right) }\left( x;-z\right) -\frac{\mathbb{W}^{\left( q\right) }\left( x\right) }{\mathbb{W}^{\left(
q\right)}\left(a\right) }w^{\left(q\right) }\left(a;-z\right)\right) \frac{z}{r}\p\left( X_{r}\in \md z\right)
\end{multline}
and
\begin{multline}\label{E:L2}
\e_{x}\left[\me^{-q\nu _{0}^{-}}\p_{Y_{\nu_{0}^{-}}}(\tau _{0}^{+}\leq r)\ind_{\left\{ \nu _{0}^{-}<\nu _{a}^{+} \right\} }\right]\\ =\int_{0}^{\infty }\left(\mathcal{W}_{x,\delta}^{\left(q,-q\right) }\left( x+z\right) -\frac{\mathbb{W}^{\left( q\right) }\left( x\right) }{\mathbb{W}^{\left(
q\right)}\left(a\right) }\mathcal{W}_{a,\delta}^{\left(q,-q\right) }\left(a+z\right) \right) \frac{z}{r}\p\left(X_{r}\in\md z\right).
\end{multline}
\end{lemma}
\begin{proof}
By \eqref{eq:lemmapart2} and Laplace inversion, we obtain, for all $y\leq 0,$
$$\e_{y}\left[ \me^{-q\tau _{0}^{+}}\ind_{\left\{ \tau _{0}^{+}\leq r\right\}
}\right] =\int_{0}^{\infty }\me^{-qr}W^{\left( q\right) }\left( y+z\right) 
\frac{z}{r}\p(X_{r}\in \md z).$$
Then, by Tonelli's theorem
\begin{align*}
&\e_{x}\left[ \me^{-q\nu _{0}^{-}}\e_{Y_{\nu _{0}^{-}}}\left[ \me^{-q\tau
_{0}^{+}}\ind_{\left\{ \tau _{0}^{+}\leq r\right\} }\right] \ind_{\left\{
\nu _{0}^{-}<\nu _{a}^{+}\right\} }\right] \\ & \qquad=\e_{x}\left[ \me^{-q\nu _{0}^{-}}\int_{0}^{\infty }\me^{-qr}W^{\left(
q\right) }\left( Y_{\nu _{0}^{-}}+z\right) \frac{z}{r}\p(X_{r}\in \md z)\ind
_{\left\{ \nu _{0}^{-}<\nu _{a}^{+}\right\} }\right] \\ & 
\qquad=\int_{0}^{\infty }\me^{-qr}\e_{x}\left[ \me^{-q\nu _{0}^{-}}W^{\left(
q\right) }\left( Y_{\nu _{0}^{-}}+z\right) \ind_{\left\{ \nu _{0}^{-}<\nu
_{a}^{+}\right\} }\right] \frac{z}{r}\p(X_{r}\in \md z)\\ & 
\qquad=\int_{0}^{\infty }\me^{-qr}\e_{x+z}\left[ \me^{-q\nu _{z}^{-}}W^{\left(
q\right) }\left( Y_{\nu _{z}^{-}}\right) \ind_{\left\{ \nu _{z}^{-}<\nu_{a+z}^{+}\right\} }\right]\frac{z}{r}\p(X_{r}\in \md z),
\end{align*}
where the last line follows by spatial homogeneity of $Y$. Using identity \eqref{eq:exp_scale} for $p=q$, we have 
$$\e_{x+z}\left[ \me^{-q\nu _{z}^{-}}W^{\left(
q\right) }\left( Y_{\nu _{z}^{-}}\right) \ind_{\left\{ \nu _{z}^{-}<\nu_{a+z}^{+}\right\} }\right] = w^{\left( q\right) }\left( x;-z\right) -
\frac{\mathbb{W}^{\left( q\right) }\left( x\right) }{\mathbb{W}^{\left(
q\right) }\left( a\right) }w^{\left( q\right) }\left( a;-z\right),
$$
which proves~\eqref{E:L1}.

By \eqref{eq:lemmapart5}, Tonelli's theorem and spatial homogeneity of $Y$, we have 
\begin{align*}
&\e_{x}\left[\me^{-q\nu _{0}^{-}}\p_{Y_{\nu_{0}^{-}}}(\tau _{0}^{+}\leq r)\ind_{\left\{ \nu _{0}^{-}<\nu _{a}^{+} \right\} }\right]=\e_{x}\left[ \me^{-q\nu _{0}^{-}}\int_{0}^{\infty }W\left( Y_{\nu _{0}^{-}}+z\right) \frac{z}{r}\p(X_{r}\in \md z)\ind_{\left\{ \nu _{0}^{-}<\nu _{a}^{+}\right\} }\right]  \\ & 
\qquad=\int_{0}^{\infty }\e_{x}\left[ \me^{-q\nu _{0}^{-}}W\left( Y_{\nu _{0}^{-}}+z\right) \ind_{\left\{ \nu _{0}^{-}<\nu
_{a}^{+}\right\} }\right] \frac{z}{r}\p(X_{r}\in \md z) \\ & \qquad=\int_{0}^{\infty }\e_{x+z}\left[ \me^{-q\nu _{z}^{-}}W\left( Y_{\nu _{z}^{-}}\right) \ind_{\left\{ \nu _{z}^{-}<\nu_{a+z}^{+}\right\} }\right]\frac{z}{r}\p(X_{r}\in \md z)
\\ & \qquad=\int_{0}^{\infty }\left(\mathcal{W}_{x,\delta}^{\left(q,-q\right) }\left( x+z\right) -\frac{\mathbb{W}^{\left( q\right) }\left( x\right) }{\mathbb{W}^{\left(
q\right)}\left(a\right) }\mathcal{W}_{a,\delta}^{\left(q,-q\right) }\left(a+z\right) \right) \frac{z}{r}\p\left(X_{r}\in\md z\right).
\end{align*}

To prove the last identity, we need to compute the following limit 
\begin{align*}
\e_{x}\left[\p_{Y_{\nu^{-}_{0}}}\left(\tau^{+}_{0}\leq r\right)\ind_{\left\lbrace \nu^{-}_{0}<\infty\right\rbrace}\right]=\lim_{q\rightarrow0 }\lim_{a\rightarrow \infty }\e_{x}\left[ \me^{-q\nu _{0}^{-}}\e_{Y_{\nu_{0}^{-}}}\left[ \me^{-q\tau _{0}^{+}}\ind_{\left\{ \tau _{0}^{+}\leq r\right\}}\right]\ind_{\left\{ \nu _{0}^{-}<\nu _{a}^{+} \right\} }\right].
\end{align*}
Since 
\begin{equation}\nonumber
\lim_{a\rightarrow \infty }\frac{W^{\left( q\right) }\left( z+a\right) }{
\WW^{\left( q\right) }\left(a\right) }=0\quad \textrm{and} \quad \lim_{a\rightarrow \infty }\frac{\WW^{\left( q\right)}\left(
a-y\right) }{\WW^{\left( q\right) }\left(a\right) }=\me^{-\varphi \left( q\right)y}.
\end{equation} 
We obtain using Lebesgue's dominated convergence theorem
\begin{align*}
&\lim_{a\rightarrow \infty }\frac{w^{\left( q\right) }\left(a;-z\right) }{
\WW^{\left( q\right) }\left(a\right) }=\delta \int_{0}^{\infty}\me^{-\varphi \left( q\right)y}W^{\left( q\right) \prime}\left( y+z\right) \md y\\ &
\qquad=-\delta W^{\left( q\right) }\left( z\right)+ \delta \me^{\varphi \left( q\right)z}\left( \frac{1}{\delta}-\varphi \left( q\right)\int_{0}^{z}\me^{-\varphi \left( q\right)y}W^{\left( q\right) }\left( y\right) \md y \right),
\end{align*}
since $\psi \left( \varphi \left( q\right) \right) -q=\psi \left( \varphi\left( q\right) \right) -\delta \varphi \left( q\right) +\delta \varphi\left( q\right) -q=\delta \varphi \left( q\right).$
 Then 
\begin{align*}
&\lim_{q\rightarrow0 }\lim_{a\rightarrow \infty }\frac{w^{\left( q\right) }\left(a,-z\right) }{W^{\left( q\right) }\left(a\right) }=-\delta W(z)+ 1
\end{align*}
and the result follows.
\end{proof}

\subsection{Proof of Theorem~\ref{refractedTh}}

For $x<0$, using the strong Markov property of $U$ and the fact that it is skip-free upward, we have
$$
\p_{x}\left(\kappa^{U}_{r}=\infty\right) = \e_{x} \left[\p_{x} \left(\kappa^{U}_{r}=\infty \mid \mathcal{F}_{\kappa^{+}_{0}} \right) \ind_{\left\lbrace \kappa^{+}_{0}<\infty \right\rbrace} \right] = \p_{x} \left( \kappa^{+}_{0}\leq r \right) \p \left( \kappa^{U}_{r}=\infty \right) .
$$
Since $\left\lbrace X_{t} , t<\tau^{+}_{0} \right\rbrace$ and $\left\lbrace U_{t} , t<\kappa^{+}_{0} \right\rbrace$ have the same distribution with respect to $\p_{x}$ when $x<0$, we further have 
\begin{align}\label{E:xNeg}
\p_{x} \left( \kappa^{U}_{r}=\infty \right) = \p_{x} \left( \tau^{+}_{0}\leq r)\p(\kappa^{U}_{r}=\infty \right) .
\end{align}

For $x\geq0$, using the strong Markov property of $U$ again, the fact that $\left\lbrace Y_{t} , t<\nu^{-}_{0} \right\rbrace$ and $\left\lbrace U_{t} , t<\kappa^{-}_{0} \right\rbrace$ have the same distribution with respect to $\p_{x}$ and using~\eqref{E:xNeg}, we get
\begin{align}
\p_{x}\left(\kappa^{U}_{r}=\infty\right) &= \p_x\left(\kappa_{0}^{-}=\infty\right) + \e_{x} \left[\p_{x} \left(\kappa^{U}_{r}=\infty \mid \mathcal{F}_{\kappa^-_{0}} \right) \ind_{\left\lbrace \kappa^-_{0}<\infty \right\rbrace} \right]  \nonumber\\
&= \p_x\left(\kappa_{0}^{-}=\infty\right) + \e_{x} \left[\p_{U_{\kappa^{-}_{0}}}\left(\kappa^{U}_{r}=\infty\right)\ind_{\left\lbrace\kappa^{-}_{0}<\infty\right\rbrace} \right] \nonumber\\
&= \p_x \left( \nu^{-}_{0}=\infty \right) + \p \left( \kappa^{U}_{r}=\infty \right) \e_{x} \left[ \p_{Y_{\nu^{-}_{0}}} \left(\tau^{+}_{0}\leq r\right) \ind_{\left\lbrace\nu^{-}_{0}<\infty\right\rbrace} \right] .\label{eq:main_decomp}
\end{align}
Note that this last expression holds for all $x\in \reals$.

We will first prove the result for $x=0$. We split this part of the proof into two cases: for processes with paths of bounded variation (BV), and then for processes with paths of unbounded variation (UBV).

First, we assume $X$ and $Y$ have paths of BV. Setting $x=0$ in~\eqref{eq:main_decomp} yields
$$
\p \left(\kappa^{U}_{r}=\infty\right) = \p \left( \nu^{-}_{0}=\infty \right) + \p \left( \kappa^{U}_{r}=\infty \right) \e \left[ \p_{Y_{\nu^{-}_{0}}} \left(\tau^{+}_{0}\leq r\right) \ind_{\left\lbrace\nu^{-}_{0}<\infty\right\rbrace} \right] .
$$
Solving for $\p \left(\kappa^{U}_{r}=\infty\right)$ and using both~\eqref{E:classicalruinprobaY} and~\eqref{E:L3}, we get
\begin{equation}\label{E:Paris00}
\p(\kappa_{r}^{U}=\infty) =\frac{\left(\e[X_{1}]-\delta\right)_{+}}{\int^{\infty}_{0}\frac{z}{r}\p(X_{r}\in \md z)-\delta} ,
\end{equation}
where we used the fact that $\WW(0)>0$.

Now, if $X$ has paths of UBV, we will use the same approximation procedure as in \cites{loeffenetal2013}. We denote by $\kappa^{U}_{r,\epsilon}$ the stopping time describing the first time an excursion, starting when $U$ gets below $0$ and ending when $U$ gets back up to $\epsilon$, lasts longer than $r$. More precisely, for $\epsilon>0$, define
\begin{equation*}
\kappa^{U}_{r,\epsilon} = \inf \left\lbrace t>r: t-g^{U}_{t,\epsilon}>r, U_{t-r}< 0 \right\rbrace,
\end{equation*}
where $g^{U}_{t,\epsilon} = \sup \left\lbrace 0\leq s\leq t \colon U_{s} \geq \epsilon \right\rbrace$. Clearly, we have $\kappa^{U}_{r,\epsilon}<\kappa^{U}_r$ a.s.\ which implies that $\left\lbrace \kappa^{U}_{r,\epsilon} = \infty \right\rbrace \subseteq \left\lbrace \kappa^{U}_r = \infty \right\rbrace$. Then, it can be shown that $\lim_{\epsilon \to 0} \p_{\epsilon} \left( \kappa^{U}_{r,\epsilon}=\infty\right) = \p \left(\kappa^{U}_{r} = \infty\right)$.

Using similar arguments as in the BV case, when $x<0$, we have
$$
\p_{x}\left(\kappa^{U}_{r,\epsilon}=\infty\right) = \p_{x}(\kappa^{+}_{\epsilon}\leq r) \p_{\epsilon}(\kappa^{U}_{r,\epsilon}=\infty)
$$
and then, when $x\geq0$, we have
$$
\p_{x}\left(\kappa^{U}_{r,\epsilon}=\infty\right) = \p_{x}\left(\nu^{-}_{0}=\infty\right) + \p_{\epsilon}(\kappa^{U}_{r,\epsilon}=\infty) \e_{x} \left[\p_{Y_{\nu^{-}_{0}}}\left(\kappa^{+}_{\epsilon} \leq r \right) \ind_{\left\lbrace\nu^{-}_{0}<\infty\right\rbrace} \right] .
$$
Setting $x=\epsilon$ and solving for $\p_{\epsilon}\left(\kappa^{U}_{r,\epsilon}=\infty\right)$, we get with the help of~\eqref{E:classicalruinprobaY}
\begin{equation}\label{UVeps}
\p_{\epsilon}\left(\kappa^{U}_{r,\epsilon}=\infty\right) = \frac{\left( \e \left[ X_1\right]-\delta \right)_{+} \WW(\epsilon)}{1-\e_{\epsilon}\left[\p_{Y_{\nu^{-}_{0}}}\left(\kappa^{+}_{\epsilon}\leq r\right) \ind_{\left\lbrace\nu^{-}_{0}<\infty\right\rbrace} \right]} .
\end{equation}

Using~\eqref{eq:first-passage_above} and then~\eqref{laplruinovershoot}, we can write
\begin{align*} 
\int^{\infty}_{0}\me^{-\theta r}\e_{\epsilon} & \left[\p_{Y_{\nu^{-}_{0}}}\left(\kappa^{+}_{\epsilon}\leq r\right)\ind_{\left\lbrace \nu^{-}_{0}<\infty\right\rbrace}\right]\md r \\
&=\e_{\epsilon}\left[\ind_{\left\lbrace \nu^{-}_{0}<\infty\right\rbrace}\int^{\infty}_{0}\me^{-\theta r}\p_{Y_{\nu^{-}_{0}}}(\kappa^{+}_{\epsilon}\leq r)\md r\right]\\
&= \frac{1}{\theta} \e_{\epsilon} \left[\ind_{\left\lbrace \nu^{-}_{0}<\infty\right\rbrace} \e_{Y_{\nu^{-}_{0}}}\left[\mathrm{e}^{-\theta \kappa^{+}_{\epsilon}} \ind_{\{\kappa^{+}_{\epsilon}<\infty\}} \right] \right]\\
&= \frac{\e_{\epsilon} \left[\ind_{\left\lbrace \nu^{-}_{0}<\infty \right\rbrace} \mathrm{e}^{\Phi(\theta) Y_{\nu^-_0}}\right]}{\theta \left( \mathrm{e}^{\Phi(\theta) \epsilon} + \delta \Phi(\theta) \int_0^\epsilon \mathrm{e}^{\Phi(\theta)y} \mathbb{W}^{(\theta)}(\epsilon-y) \mathrm{d}y \right)} \\
&=\frac{1-(\theta-\delta\Phi(\theta)) \int^{\epsilon}_{0} \me^{-\Phi(\theta) y}\WW(y)\md y - \frac{\theta -\delta \Phi(\theta)}{\Phi(\theta)} \me^{-\Phi(\theta)\epsilon} \WW(\epsilon)}{\theta \left(1+\delta \Phi(\theta) \int_0^{\epsilon} \me^{-\Phi(\theta)y} \WW^{(\theta)}(y) \mathrm{d}y \right)} .
\end{align*}
Consequently, we have
\begin{align*}
\int^{\infty}_{0}\me^{-\theta r} & \left( \frac{1-\e_{\epsilon}\left[\p_{Y_{\nu^{-}_{0}}}\left(\kappa^{+}_{\epsilon}\leq r\right)\ind_{\left\lbrace \nu^{-}_{0}<\infty\right\rbrace}\right]}{\WW(\epsilon)} \right) \md r \\
&= \frac{1}{\theta \WW(\epsilon)} - \frac{1-(\theta-\delta\Phi(\theta)) \int^{\epsilon}_{0} \me^{-\Phi(\theta) y}\WW(y)\md y - \frac{\theta -\delta \Phi(\theta)}{\Phi(\theta)} \me^{-\Phi(\theta)\epsilon} \WW(\epsilon)}{\theta \WW(\epsilon) \left(1+\delta \Phi(\theta) \int_0^{\epsilon} \me^{-\Phi(\theta)y} \WW^{(\theta)}(y) \mathrm{d}y \right)} \\
&= \frac{1}{\theta \WW(\epsilon)} \left( \frac{\delta \Phi(\theta) \int_0^{\epsilon} \me^{-\Phi(\theta)y} \WW^{(\theta)}(y) \mathrm{d}y + (\theta-\delta\Phi(\theta)) \int^{\epsilon}_{0} \me^{-\Phi(\theta) y}\WW(y)\md y + \frac{\theta -\delta \Phi(\theta)}{\Phi(\theta)} \me^{-\Phi(\theta)\epsilon} \WW(\epsilon)}{1+\delta \Phi(\theta) \int_0^{\epsilon} \me^{-\Phi(\theta)y} \WW^{(\theta)}(y) \mathrm{d}y} \right) \\
& \xrightarrow[\epsilon \to 0]{} \frac{1}{\Phi(\theta)}-\frac{\delta}{\theta} ,
\end{align*} 
where we used the fact that, for all $\theta \geq 0$, we have
\begin{equation*}
\lim_{\epsilon\rightarrow0} \frac{\int_0^{\epsilon} \me^{-\Phi(\theta) y} \WW^{(\theta)}(y) \mathrm{d}y}{\WW(\epsilon)} = 0 .
\end{equation*} 
From~\eqref{eq:lemmapart1} in Lemma~\ref{laplace}, we have that $\theta \mapsto 1/\Phi(\theta)-\delta/\theta$ is the Laplace transform of
$$
r \mapsto \int_{0}^\infty \frac{z}{r} \p \left( X_r \in \mathrm{d}z \right) - \delta .
$$
By the Extended continuity theorem of Laplace transforms (see e.g.\ \cite{fellervol2}), this concludes the proof for $x=0$.

We now prove the result for $x \in \mathbb{R}$. Now, $X$ and $Y$ can be of BV or of UBV. We can now write~\eqref{eq:main_decomp} as follows:
\begin{align*}
\p_{x}(\kappa_{r}^{U}=\infty) &= \left(\e\left[X_{1}\right]-\delta\right)_{+} \WW(x) + \frac{(\e[X_{1}]-\delta)_{+}}{\int^{\infty}_{0}\frac{z}{r}\p(X_{r}\in \md z)-\delta} \e_{x}\left[\p_{Y_{\nu^{-}_{0}}}(\tau^{+}_{0}\leq r)\ind_{\left\lbrace\nu^{-}_{0}<\infty\right\rbrace}\right] \\
&= \left(\e\left[ X_{1}\right]-\delta\right)_{+}\left(\frac{\WW(x) \left(\int^{\infty}_{0}\frac{z}{r}\p(X_{r}\in \md z)-\delta \right)+\e_{x} \left[ \p_{Y_{\nu^{-}_{0}}}\left(\tau^{+}_{0}\leq r \right) \ind_{\left\lbrace \nu^{-}_{0}<\infty \right\rbrace} \right]}{\int^{\infty}_{0} \frac{z}{r} \p(X_{r}\in \md z) - \delta} \right) .
\end{align*}
Using~\eqref{E:L3}, we get finally
$$
\p_{x}\left(\kappa_{r}^{U}=\infty \right) = \left(\e[X_{1}]-\delta\right)_{+}\left(\frac{\int^{\infty}_{0}w(x;-z)z\p(X_{r}\in \md z)}{\int^{\infty}_{0}z\p(X_{r}\in \md z)-\delta r} \right) ,
$$
which holds for all $x\in \reals$. \begin{flushright} $\blacksquare$ \end{flushright}

\subsection{Proof of Theorem~\ref{refracted}}

For $x<0$, using the strong Markov property of $U$ and the fact that it is skip-free upward, we have
\begin{eqnarray*}
\e_{x}\left[ \me^{-q\kappa _{r}^{U}}\ind_{\left\{ \kappa
_{r}^{U}<\kappa _{a}^{+}\right\} }\right]=\me^{-qr}\p_{x}(\kappa _{0}^{+}>r)+\e_{x}\left[ \me^{-q\kappa_{0}^{+}}\ind_{\left\{ \kappa _{0}^{+}\leq r\right\} }\right] \e\left[
\me^{-q\kappa _{r}^{U}}\ind_{\left\{ \kappa _{r}^{U}<\kappa_{a}^{+}\right\} }\right].
\end{eqnarray*}
Since$\left\lbrace X_{t},\text{ }t<\tau^{+}_{0}\right\rbrace$ and $\left\lbrace U_{t},\text{ }t<\kappa^{+}_{0}\right\rbrace$ have the same law under $\p_{x}$ when $x<0$, we obtain
\begin{align}\label{E:xNegb}
\e_{x}\left[ \me^{-q\kappa _{r}^{U}}\ind_{\left\{ \kappa
_{r}^{U}<\kappa _{a}^{+}\right\} }\right]=\me^{-qr}\p_{x}(\tau _{0}^{+}>r)+\e_{x}\left[ 
\me^{-q\tau _{0}^{+}}\ind_{\left\{ \tau _{0}^{+}\leq r\right\} }\right]
\e\left[ \me^{-q\kappa _{r}^{U}}\ind_{\left\{ \kappa _{r}^{U}<\kappa_{a}^{+}\right\} }\right] .
\end{align}
For $0\leq x\leq a$, using the strong Markov property again, we get
\begin{eqnarray*}
\e_{x}\left[ \me^{-q\kappa _{r}^{U}}\ind_{\left\{ \kappa
_{r}^{U}<\kappa _{a}^{+}\right\} }\right]=
\e_{x}\left[ \me^{-q\kappa_{0}^{-}}\e_{U_{\kappa _{0}^{-}}}\left[ \me^{-q\kappa _{r}^{U}}\ind_{\left\{ \kappa_{r}^{U}<\kappa _{a}^{+}\right\} }\right]\ind_{\left\{ \kappa _{0}^{-}< \kappa _{a}^{+}\right\} } \right] 
\end{eqnarray*}
Using the fact that $\left\lbrace Y_{t},\text{ }t<\nu^{-}_{0}\right\rbrace$ and $\left\lbrace U_{t},\text{ }t<\kappa^{-}_{0}\right\rbrace$ have the same law under $\p_{x}$ when $x\geq 0$ and injecting~\eqref{E:xNegb} in the last expectation, we have, for all $x\in\reals$
\begin{align*}
&\e_{x}\left[ \me^{-q\kappa _{r}^{U}}\ind_{\left\{ \kappa
_{r}^{U}<\kappa _{a}^{+}\right\} }\right]
=\me^{-qr}\e_{x}\left[ \me^{-q\nu _{0}^{-}}\ind_{\left\{ \nu_{0}^{-}<\nu_{a}^{+} \right\} }\right] -\me^{-qr}\e_{x}\left[\me^{-q\nu _{0}^{-}}\p_{Y_{\nu_{0}^{-}}}(\tau _{0}^{+}\leq r)\ind_{\left\{ \nu _{0}^{-}<\nu _{a}^{+} \right\} }\right]\\
& \qquad+\e\left[ \me^{-q\kappa _{r}^{U}}\ind_{\left\{ \kappa
_{r}^{U}<\kappa _{a}^{+}\right\} }\right] \e_{x}\left[ \me^{-q\nu _{0}^{-}}\e_{Y_{\nu
_{0}^{-}}}\left[ \me^{-q\tau _{0}^{+}}\ind_{\left\{ \tau _{0}^{+}\leq r\right\}
}\right]\ind_{\left\{ \nu _{0}^{-}<\nu _{a}^{+} \right\} }\right] .
\end{align*}
For $x=0$ and using the last equation
\begin{equation*}
\e\left[ \me^{-q\kappa _{r}^{U}}\ind_{\left\{ \kappa_{r}^{U}<\kappa _{a}^{+}\right\} }\right] =\frac{\me^{-qr}\e\left[ \me^{-q\nu _{0}^{-}}\ind_{\left\{ \nu _{0}^{-}<\nu _{a}^{+} \right\}}
\right]-\me^{-qr}\e\left[\me^{-q\nu _{0}^{-}}\p_{Y_{\nu_{0}^{-}}}(\tau _{0}^{+}\leq r)\ind_{\left\{ \nu _{0}^{-}<\nu _{a}^{+} \right\} }\right] }{1-\e\left[ \me^{-q\nu _{0}^{-}}\e_{Y_{\nu_{0}^{-}}}\left[ \me^{-q\tau _{0}^{+}}\ind_{\left\{ \tau _{0}^{+}\leq r\right\}
}\right]\ind_{\left\{ \nu _{0}^{-}<\nu _{a}^{+} \right\} }\right] }
\end{equation*}
where, from~\eqref{E:L1}, 
\begin{align*}
&\e\left[ \me^{-q\nu _{0}^{-}}\e_{Y_{\nu_{0}^{-}}}\left[ \me^{-q\tau _{0}^{+}}\ind_{\left\{ \tau _{0}^{+}\leq r\right\}
}\right]\ind_{\left\{ \nu _{0}^{-}<\nu _{a}^{+} \right\} }\right]\\
& \qquad=\int_{0}^{\infty }\me^{-qr}\left(W^{\left(q\right) }\left(z\right) -\frac{\mathbb{W}^{\left( q\right) }\left( 0\right) }{\mathbb{W}^{\left(
q\right)}\left(a\right) }w^{\left(q\right) }\left(a;-z\right)\right) \frac{z}{r}\p\left( X_{r}\in \md z\right)
\end{align*}
and, from~\eqref{E:L2},
\begin{align*}
&\e\left[\me^{-q\nu _{0}^{-}}\p_{Y_{\nu_{0}^{-}}}(\tau _{0}^{+}\leq r)\ind_{\left\{ \nu _{0}^{-}<\nu _{a}^{+} \right\} }\right]\\
& \qquad=\int_{0}^{\infty }\left(W\left(z\right)-\frac{\mathbb{W}^{\left( q\right) }\left(0\right) }{\mathbb{W}^{\left(
q\right)}\left(a\right) }\mathcal{W}_{a,\delta}^{\left(q,-q\right) }\left(a+z\right) \right) \frac{z}{r}\p\left( X_{r}\in \md z\right).
\end{align*}

With the help of~\eqref{eq:lemmapart3}, of~\eqref{eq:lemmapart4} with $y=0$, and since $\WW(0)>0$, we obtain
\begin{align}\nonumber
\e\left[ \me^{-q\kappa _{r}^{U}}\ind_{\left\{ \kappa
_{r}^{U}<\kappa _{a}^{+}\right\} }\right] =\frac{-\me^{-qr}\frac{\mathbb{W}^{\left( q\right) }\left( 0\right) }{\mathbb{W}^{\left(
q\right)}\left(a\right) }\mathbb{Z}^{\left(
q\right)}\left(a\right)
+\me^{-qr}\int_{0}^{\infty }\frac{\mathbb{W}^{\left( q\right) }\left(0\right) }{\mathbb{W}^{\left(q\right)}\left(a\right) }\mathcal{W}_{a,\delta}^{\left(q,-q\right) }\left(a+z\right) \frac{z}{r}\p\left( X_{r}\in \md z\right)
}{\frac{\mathbb{W}^{\left( q\right) }\left( 0\right) }{\mathbb{W}^{\left(
q\right)}\left(a\right) }\int_{0}^{\infty }\me^{-qr}w^{\left(q\right) }\left(a;-z\right)\frac{z}{r}\p\left( X_{r}\in \md z\right)
} \\ \label{zeropar}
=1-\frac{\mathbb{Z}^{\left(
q\right)}\left(a\right)
+\int_{0}^{\infty }\left(w^{\left(q\right) }\left(a;-z\right)
-\mathcal{W}_{a,\delta}^{\left(q,-q\right) }\left(a+z\right) \right)\frac{z}{r}\p\left( X_{r}\in \md z\right)
}
{
\int_{0}^{\infty }w^{\left(q\right) }\left(a;-z\right)\frac{z}{r}\p\left( X_{r}\in \md z\right)
} .
\end{align}

Then 
\begin{align*}
&\me^{qr}\e_{x}\left[ \me^{-q\kappa _{r}^{U}}\ind_{\left\lbrace \kappa_{r}^{U}<\kappa _{a}^{+}\right\rbrace }\right]=\mathbb{Z}^{\left(
q\right)}\left(x\right)-\mathbb{Z}^{\left(
q\right)}\left(a\right)\frac{\mathbb{W}^{\left( q\right) }\left(x\right) }{\mathbb{W}^{\left(q\right)}\left(a\right) }\\
& \qquad-\int_{0}^{\infty }\left(\mathcal{W}_{x,\delta}^{\left(q,-q\right) }\left( x+z\right) -\frac{\mathbb{W}^{\left( q\right) }\left( x\right) }{\mathbb{W}^{\left(
q\right)}\left(a\right) }\mathcal{W}_{a,\delta}^{\left(q,-q\right) }\left(a+z\right) \right) \frac{z}{r}\p\left( X_{r}\in \md z\right)\\
& \qquad+\e\left[ \me^{-q\kappa _{r}^{U}}\ind_{\left\lbrace \kappa
_{r}^{U}<\kappa _{a}^{+}\right \rbrace }\right] \int_{0}^{\infty }\left(w^{\left(q\right) }\left( x;-z\right) -\frac{\mathbb{W}^{\left( q\right) }\left( x\right) }{\mathbb{W}^{\left(
q\right)}\left(a\right) }w^{\left(q\right) }\left(a;-z\right)\right) \frac{z}{r}\p\left( X_{r}\in \md z\right)
\\
& \qquad=\mathbb{Z}^{\left(q\right)}\left(x\right)+\int_{0}^{\infty}\left( w^{\left(q\right)}\left(x;-z\right)\e\left[\me^{-q\kappa^{U}_{r}}\ind_{\left\lbrace \kappa^{U}_{r}<\kappa_a^{+}\right\rbrace}\right]
-\mathcal{W}_{x,\delta}^{\left( q,-q\right) }\left(x+z\right)
 \right)\frac{z}{r}\p\left( X_{r}\in \md z\right).
\end{align*}

When $X$ has paths of unbounded variation, we can use the same approximation procedure as in the proof of Theorem~\ref{refractedTh}. The details are left to the reader.

Identity (ii) follows from (i) by taking limit. Indeed, we have
\begin{align*}
\lim_{a\rightarrow \infty }& \e_{x}\left[ \mathrm{e}^{-q(\kappa _{r}^{U}-r)}%
\ind_{\left\{ \kappa _{r}^{U}<\kappa _{a}^{+}\right\} }\right]  =\lim_{a\rightarrow
\infty }\e\left[ \me^{-q\kappa _{r}^{U}}\ind_{\left\{ \kappa _{r}^{U}<\kappa
_{a}^{+}\right\} }\right] \int_{0}^{\infty
} w^{\left( q\right) }\left( x;-z\right) \frac{z}{r}\p\left( X_{r}\in \md z\right)\\
& \qquad + \mathbb{Z}^{\left( q\right) }\left( x\right)  -
\int_{0}^{\infty
}\mathcal{W}_{x,\delta}^{\left( q,-q\right)
}\left( x+z\right) \frac{z}{r}\p\left( X_{r}\in \md z\right) ,
\end{align*}
and, from~\eqref{zeropar},
\begin{eqnarray*}
\lim_{a\rightarrow \infty }\e\left[ \me^{-q\kappa _{r}^{U}}\ind_{\left(
\kappa _{r}^{U}<\kappa _{a}^{+}\right) }\right]  &=&\lim_{a\rightarrow
\infty }\frac{-\mathbb{Z}^{\left(
q\right)}\left(a\right)
+\int_{0}^{\infty }\mathcal{W}_{a,\delta}^{\left(q,-q\right) }\left(a+z\right) \frac{z}{r}\p\left( X_{r}\in \md z\right)
}{\int_{0}^{\infty }w^{\left(q\right) }\left(a;-z\right)\frac{z}{r}\p\left( X_{r}\in \md z\right)
} .
\end{eqnarray*}
As shown before, we have
\begin{equation*}
\lim_{a\rightarrow \infty }\dfrac{w^{(q)}(a;-z)}{\WW^{(q)}(a)}=-\delta W^{\left(
q\right) }\left( z\right) +\me^{\varphi \left( q\right) z}\left( 1-\delta
\varphi \left( q\right) \int_{0}^{z}\me^{-\varphi \left( q\right)
y}W^{\left( q\right) }\left( y\right) \md y\right) ,
\end{equation*}
Then
\begin{align*}
&\lim_{a\rightarrow \infty } \int_{0}^{\infty }\dfrac{w^{(q)}(a;-z)}{\WW^{(q)}(a)}\frac{z}{r}\p\left( X_{r}\in \md z\right)\\ & \qquad
=\int_{0}^{\infty }\left( 1-\delta \varphi \left( q\right) \int_{0}^{z}%
\mathrm{e}^{-\varphi \left( q\right) v}W^{\left( q\right) }\left( v\right)
\mathrm{d}v\right) \me^{\varphi \left( q\right) z}\frac{z}{r}\p\left( X_{r}\in \md %
z\right) -\delta \mathrm{e}^{qr}.
\end{align*}
Finally, from the definition of $\mathcal{W}_{a,\delta}^{\left(q,-q\right)} $, using Lebesgue's dominated convergence theorem and performing an integration by parts,
\begin{eqnarray*}
&&\lim_{a\rightarrow \infty } \frac{-\mathbb{Z}^{\left(
q\right)}\left(a\right)
+\int_{0}^{\infty }\mathcal{W}_{a,\delta}^{\left(q,-q\right) }\left(a+z\right) \frac{z}{r}\p\left( X_{r}\in \md z\right)
}{\WW^{(q)}(a)}\\
&=&-\dfrac{q}{\varphi \left( q\right) }+\lim_{a\rightarrow \infty }\int_{0}^{\infty }\left(\dfrac{\mathbb{W}^{\left(
q\right) }\left( a+z\right)-\delta W(z)\mathbb{W}^{\left(
q\right) }\left( a\right)}{\WW^{\left(
q\right)}(a)}\right) \frac{z%
}{r}\p\left( X_{r}\in \md z\right)  \\
&&+\lim_{a\rightarrow \infty }\int_{0}^{\infty }\frac{z}{r}
\p(X_{r}\in \md z)\int_{0}^{z}\left(q W\left( z-y\right)  -\delta W'\left( z-y\right) \right)\frac{\WW^{\left( q\right) }\left( a+y\right)}{\WW^{\left( q\right) }\left( a\right) } \md y
\\
&=&-\dfrac{q}{\varphi \left( q\right) }-\delta+\int_{0}^{\infty } \me^{\varphi \left( q\right) z}\left(1+ \left(q-\delta \varphi \left( q\right) \right)\int_{0}^{z} \me^{-\varphi
\left( q\right) y}W(y)\mathrm{d}y \right)\frac{z}{r}
\p(X_{r}\in \md z) .
\end{eqnarray*}

To prove (iii), we use first the strong Markov property and the fact that $U$ has only downward jumps to get
\begin{equation}\label{twosided1}\nonumber
\p_x(\kappa^U_r=\infty)=\p_x(\kappa_a^+<\kappa^U_r)\p_a(\kappa^U_r=\infty) ,
\end{equation}
which yields
$$
\p_x(\kappa_a^+<\kappa^U_r)=\frac{\p_x(\kappa^U_r=\infty)}{\p_a(\kappa^U_r=\infty)} .
$$
Using the change of measure $\frac{\md\p_{x}^{\Phi(q)}}{\md\p_{x}} =\mathrm{e}^{\Phi(q) (X_{t}-x)-q t}$ on $\mathcal{F}_{t}$ and using~\eqref{refractedPr}, we get
\begin{align*}
V^{(q)}(x) &:= \mathrm{e}^{\Phi(q)x}\p^{\Phi(q)}_x(\kappa^U_r=\infty) \\
&= \left(\e^{\Phi(q)}\left[ X_1 \right]-\delta\right)_{+}\frac{\int_0^\infty  \mathrm{e}^{-\Phi(q)z} w^{(q)}(x;-z)  z\mathbb{P}^{\Phi(q)}(X_{r}\in\mathrm{d}z) }{ \int_0^\infty   z\mathbb{P}^{\Phi(q)}(X_{r}\in\mathrm{d}z)-\delta r} \\
&= \left(\mathbb \e^{\Phi(q)} \left[X_1\right] -\delta\right)_{+}\frac{\int_0^\infty  w^{(q)}(x;-z)  z\mathbb{P}(X_{r}\in\mathrm{d}z) }{ \int_0^\infty   z\mathbb{P}^{\Phi(q)}(X_{r}\in\mathrm{d}z)-\delta r}.
\end{align*}
Consequently, from the Optional Stopping Theorem and from the fact that, with respect to $\p^{\Phi(q)}$, $X$ and $Y$ drift to infinity (since $\psi_{\Phi(q)}^\prime(0+)=\psi^\prime(\Phi(q)+)>0$), we have
\begin{equation}\label{parisiantwosided}\nonumber
\e_x\left[\mathrm{e}^{-q\kappa^+_a},\kappa_a^+<\kappa^U_r\right] = \frac{V^{(q)}(x)}{V^{(q)}(a)} = \frac{\int_0^\infty  w^{(q)}(x;-z)  z\mathbb{P}(X_{r}\in\mathrm{d}z)}{\int_0^\infty  w^{(q)}(a;-z)  z\mathbb{P}(X_{r}\in\mathrm{d}z)}. 
\end{equation} 
\begin{flushright} $\blacksquare$ \end{flushright}

\appendix
\section{A few analytical properties of scale functions}

The $q$-scale function $W^{(q)}$, of a spectrally negative Lévy process $X$, is differentiable except for at most countably many points. Moreover, $W^{(q)}$ is continuously differentiable if $X$ has paths of unbounded variation or if the tail of the L\'evy measure is continuous, and it is twice continuously differentiable on $(0,\infty)$ if $\sigma>0$. The initial values of $W^{(q)}$ and $W^{(q)\prime}$ are given by
\begin{equation*}\label{initialvalues}
\begin{split}
W^{(q)}(0+) &=
\begin{cases}
1/\drift & \text{when $\sigma=0$ and $\int_{0}^1 z \Pi(\mathrm{d}z) < \infty$,} \\
0 & \text{otherwise,}
\end{cases}\\ 
W^{(q)\prime}(0+) &=
\begin{cases}
2/\sigma^2 & \text{when $\sigma>0$,} \\
(\Pi(0,\infty)+q)/c^2 & \text{when $\sigma=0$ and $\Pi(0,\infty)<\infty$,} \\
\infty & \text{otherwise.}
\end{cases}
\end{split}
\end{equation*}
On the other hand, when $\psi'(0+) > 0$, the \textit{terminal value} of $W$ is given by
$$
\lim_{x \to \infty} W(x) = \frac{1}{\psi'(0+)}.
$$
It is also well known that
\begin{equation}\label{limitZ&W}\nonumber
\lim_{x \to \infty} \frac{Z^{(q)}(x)}{W^{(q)}(x)}=\frac{q}{\Phi(q)}.
\end{equation}

Finally, recall the following useful identity taken from \cite{renaud2014}: for $p,q \geq 0$ and $x \in \mathbb{R}$,
\begin{multline}\label{E:convolution}
(q-p) \int_0^x \WW^{(p)}(x-y) W^{(q)}(y) \mathrm{d}y \\
= W^{(q)}(x) - \WW^{(p)}(x) + \delta \left( W^{(q)}(0) \mathbb{W}^{(p)}(x) + \int_0^x \WW^{(p)}(x-y) W^{(q) \prime}(y) \mathrm{d}y \right) ,
\end{multline}
where $\WW^{(q)}$ is the $q$-scale function of the spectrally negative Lévy process $Y=\{Y_t=X_t-\delta t, t\geq 0\}$. Note that when $\delta=0$, we recover a special case first obainted in \cite{loeffenetal2014}:
\begin{equation}\label{eq:sym} 
(q-p)\int_0^x\pscale(x-y)\qscale(y)\mathrm{d}y=\qscale(x)-\pscale(x) .
\end{equation}

\section{Acknowledgements}

We thank two anonymous referees for their careful reading of the paper.

Funding in support of this work was provided by the Natural Sciences and Engineering Research Council of Canada (NSERC).

Mohamed Amine Lkabous thanks the Institut des sciences math\'ematiques (ISM) and the Faculté des sciences at UQAM for their financial support (PhD scholarships).

Irmina Czarna is supported by National Science Centre Grant No. 2015/19/D/ST1/01182.

\bibliographystyle{alpha}
\bibliography{REFERENCES}
\end{document}